\documentclass[a4paper,12pt]{amsart}

\usepackage{amssymb,xspace,amscd,yhmath,mathdots,txfonts,
mathrsfs, youngtab,stmaryrd}

\usepackage[matrix,arrow,curve]{xy}\CompileMatrices

\usepackage{hyperref}
\usepackage{tikz-cd}
\usepackage{tikz}
\usepackage{yfonts}[1998/10/03]

\DeclareMathAlphabet{\mathpzc}{OT1}{pzc}{m}{it}

\numberwithin{equation}{section}

\theoremstyle{plain}

\newtheorem{thm}{Theorem}[section]

\newtheorem{lem}[thm]{Lemma}

\newtheorem{cor}[thm]{Corollary}

\newtheorem{prop}[thm]{Proposition}

\theoremstyle{definition}

\newtheorem{defn}{Definition}[section]
\newtheorem{exam}[thm]{Example}

\newtheorem{rmk}[thm]{Remark}

\newcommand\Fb{\mathbb{F}}

\DeclareMathAlphabet{\mathpzc}{OT1}{pzc}{m}{it}

\newcommand\ot{\mathfrak{o}}

\DeclareMathOperator{\R}{{\mathbb{R}}}

\DeclareMathOperator{\rank}{rank}

\DeclareMathOperator{\Li}{\mathfrak{L}}

\DeclareMathOperator{\Gf}{\mathbf{G}}

\DeclareMathOperator{\zt}{\mathfrak{z}}
\DeclareMathOperator{\st}{\mathfrak{s}}

\DeclareMathOperator{\N}{\mathrm{N}}
\DeclareMathOperator{\Z}{\mathbb{Z}}

\DeclareMathOperator{\at}{\mathfrak{a}}
\DeclareMathOperator{\pt}{\mathfrak{p}}
\DeclareMathOperator{\ft}{\mathfrak{f}}

\DeclareMathOperator{\C}{\mathbb{C}}

\newcommand\GL{\mathbf{GL}}

\newcommand\gt{\mathfrak{g}}

\newcommand\Rad{\mathpzc{R}}
\newcommand\hg{\mathfrak{h}}

\newcommand\iq{\mathpzc{i}}

\newcommand\spt{\mathfrak{sp}}

\newcommand\Symm{\mathpzc{S}}
\newcommand\Jd{\mathrm{J}}
\newcommand\qq{\mathpzc{q}}
\newcommand\pq{\mathpzc{p}}

\newcommand\mt{\textswab{m}}

\newcommand\Ji{\mathpzc{J}}

\newcommand\Sb{\mathbf{S}}
\newcommand\Id{\mathrm{I}}

\newcommand\diag{\mathrm{diag}}

\newcommand\bt{\mathfrak{b}}

\newcommand\bq{\mathpzc{b}}

\newcommand\nt{\mathfrak{n}}

\newcommand\pg{\mathfrak{p}}

\newcommand\gl{\mathfrak{gl}}
\newcommand\trac{\mathrm{trace}}

\newcommand\slt{\mathfrak{sl}}

\newcommand\Az{\mathpzc{S}(V)}

\newcommand\Bt{\mathfrak{B}}
\newcommand\Ki{\mathpzc{K}}

\newcommand\vq{\mathpzc{v}}
\newcommand\wq{\mathpzc{w}}
\newcommand\uq{\mathpzc{u}}
\newcommand\T{\mathrm{T}}

\newcommand\Lt{\mathfrak{L}}
\newcommand\Ft{\mathfrak{F}}

\newcommand\Der{\mathpzc{Der}}

\newcommand{\pnt}{{\mathrm{p}}}
\newcommand\aq{\mathpzc{a}}

\newcommand\eq{\mathpzc{e}}
\newcommand\pri{\mathfrak{p}}
\newcommand\Mi{\mathpzc{M}}
\newcommand\M{\mathrm{M}}

\newcommand\Xx{\mathpzc{X}}
\newcommand\epi{\epsilonup}
\newcommand\zq{\mathpzc{z}}
\newcommand\Ann{\mathpzc{Ann}}
\newcommand\sq{\mathpzc{s}}
\newcommand\cq{\mathpzc{c}}

\newcommand\K{\mathbb{K}}

\newcommand\Nt{\mathfrak{N}}
\newcommand{\per}{\pmb{\cdot}}

\newcommand\kq{\mathrm{k}}

   \def\DHLhksqrt#1#2{\setbox0=\hbox{$#1\sqrt{#2\,}$}\dimen0=\ht0
     \advance\dimen0-0.2\ht0
     \setbox2=\hbox{\vrule height\ht0 depth -\dimen0}%
     {\box0\lower0.4pt\box2}}

\title[Prolongations of FGLA]{
$\Li$-prolongations of 
graded Lie algebras}

\author{Stefano Marini, Costantino Medori, Mauro Nacinovich}

\address{Stefano Marini: Dipartimento di Scienze Matematiche, Fisiche e Informatiche\\ 
Unit\`a di Matematica e Informatica
\\
Universit\`a degli Studi di Parma\\ Parco Area delle Scienze 53/A (Campus), 43124 Parma
 (Italy)} \email{stefano.marini@unipr.it}

\address{Costantino Medori:
Dipartimento di Scienze Matematiche, Fisiche e Informatiche\\
Unit\`a di Matematica e Informatica
\\ Universit\`a degli Studi di Parma\\ Parco Area delle Scienze 53/A (Campus), 43124 Parma
 (Italy)} \email{costantino.medori@unipr.it}

\address{Mauro Nacinovich:
Dipartimento di Matematica\\ II Universit\`a degli Studi di Roma
``Tor Ver\-ga\-ta''\\ Via della Ricerca Scientifica\\ 00133 Roma
(Italy)}
\email{nacinovi@mat.uniroma2.it}

\subjclass[2000]{Primary: 17B70, 53C10;
Secondary:  53C15, 70G65}
\keywords{$\mathbf{G}$-structure, Fundamental graded Lie algebra, Tanaka's prolongation}

\date\today

\begin{document}

\begin{abstract}
 In this paper we translate the necessary and sufficient conditions of
 Tanaka's theorem on the
 finiteness of effective prolongations of a fundamental graded Lie algebras
 into computationally effective criteria, involving  the rank 
 of some matrices that can be explicitly constructed. 
 Our results would apply to  geometries, which are defined by assigning
 a structure algebra on the contact distribution.
\end{abstract}

\maketitle

 \tableofcontents

\section*{Introduction}
The concept of $\Gf$-structure was introduced to treat various
interesting differential
geometrical structures in a unified manner (see e.g. \cite{kn1964a,kn1964,kn1965,kn1965a,kn1965b,kn1966,Kob}). 
At a chosen point $\pnt_0$ of a manifold $M,$ 
a $\Gf$-strucure 
can be 
described by the datum of a Lie algebra 
$\Li\,{=}\,{Lie(\Gf)}$ of 
infinitesimal transformations,
acting as linear maps on the tangent space $V{=}T_{\pnt_0}M.$
It is convenient to envision 
$V{\oplus}\Li$ as an Abelian \textit{extension} of $\Li$ and to look for
its maximal effective prolongation to read the differential invariants of the structure.
When this turns out to be finite dimensional, 
Cartan's method can be used to study the automorphism
group  and the equivalence problem for the corresponding $\Gf$-structure.
This \textit{algebraic} point of view was taken up systematically in \cite{GS,Sternberg},
and  
pursued 
in \cite{Tan67,Tan70,AMN06,Warhurst2007,Ottazzi2010,Kruglikov2011,Ottazzi2011a,Ottazzi2011,
alek} 
to study generalised contact 
and $CR$ structures: the datum of a smooth distribution 
defines on $T_{\pnt_0}M$ a structure of $\Z$-graded nilpotent Lie algebra 
$\mt{=}{\sum}_{\pq=1}^\muup\gt_{{-}\pq}$;
moreover, the \textit{structure group} yields 
a Lie 
algebra $\gt_{0}$
of  $0$-degree
derivations of $\mt,$
that encodes the geometry of 
$M$.
The condition that the distribution be completely non integrable,
or satisfies \textit{H\"ormander's condition} at $\pnt_0,$ 
translates into the fact that $\mt$ is \textit{fundamental}, i.e.
that the part 
$\gt_{-1},$ tangent
at $\pnt_0$ to  
the contact distribution, generates $\mt.$ 
In this case the action of the structure algebra on $\mt$
is uniquely determined by its action on $\gt_{-1}.$  Thus
we prefer to consider general Lie subalgebras $\Li$ of
$\gl_{\K}(\gt_{-1})$ and look for maximal $\Li$-prolongations 
of $\mt,$ i.e. prolongations with $\gt_{0}\,{\subseteq}\,\Li.$ 
\par
Having recently proved in 
\cite{NMSM} a finiteness result for the automorphism
group of 
a class of homogeneous
$CR$ manifolds by applying 
a result of N.Tanaka to a suitably filtered object, we got interested in
the general preliminary
problem of the finiteness of the effective 
$\Li$-prolongations of general fundamental
$\Z$-graded Lie algebras, starting from 
rereading 
Tanaka's
seminal paper \cite{Tan70}. \par 
Our main results show that, 
given a subalgebra  $\Li$ 
of $\gl_{\K}(\gt_{-1}),$ 
 the finiteness of the maximal 
$\Li$-prolongation is equivalent to a rank condition on some
matrices which can be explicitly constructed.\par 
Finiteness of Tanaka's prolongations
was already discussed 
in several different special contexts 
(see e.g. \cite{Ottazzi2011} and references therein). 
Here we tried to take
a  quite general perspective and gave  effective methods 
toward computationally dealing with this problem.
\smallskip\par 
Let us describe the contents of this work. 
Being essentially algebraic, all results 
are formulated for an arbitrary ground 
field $\K$ of characteristic zero.
\par
In \S\ref{sec-fund} we describe  
fundamental graded Lie algebras as quotients of the free Lie algebra
$\ft(V)$ generated by a vector space $V$, by its graded ideals~$\Ki.$ 
After getting from \cite{Bou68}
the maximal effective prolongation $\Ft(V)$ of $\ft(V),$
we  use this result to characterise,  in Theorem~\ref{thm-tan-1-8},
 the maximal effective $\Li$-pro\-lon\-ga\-tion of $\mt$ 
in terms of $\Ft(V)$ and $\Ki.$ \par 
In \S\ref{sect2} we review the finiteness theorem 
proved by N.Tanaka in \cite{Tan70}. Here we make explicit
the role of duality, which allows 
to reduce to commutative algebra. This was hinted in
Serre's appendix to \cite{GS} and also in \cite{Tan70}, 
but, in our opinion, in a way which left part of the arguments
rather obscure. We hope that our exposition would make
this important theorem  more understandable.
Theorem~\ref{thm-t-4-4}
reduces the finiteness of the maximal 
$\Li$-prolongation to the analogous problem 
for an $\Li'${-}prolongation of an Abelian Lie algebra.
As an application, 
we provide 
a comparison result for 
prolongations with different $\Li$'s and $\Ki$'s.
In the following sections we provide 
a criterion for studying prolongations of Abelian
Lie algebras and 
to obtain an explicit description of $\Li'$ 
after knowing $\Li$ and~$\mt.$ 
\par
In \S\ref{sec5a} we get an effective 
crierion for the finiteness of 
$\Li$-prolongations of fundamental graded Lie algebras
of the first kind, using duality to reduce  the question to the study 
the co-primary decomposition of finitely generated modules over the
polynomials: this boils down  
to computing the rank
of a matrix $M_1(\Li,\zq)$ 
of first degree homogeneous polynomials  associated to~$\Li$ 
(see Theorem~\ref{thm-t-4-6}).
In this way, a necessary and sufficient condition 
for finite prolongation
can be formulated as an ellipticity condition 
(in the sense of Kobayashi, cf. \cite[p.4]{Kob})
for $\Fb\,{\otimes}\,\Li',$
where $\Li'$ is 
the \textit{reduced} structure algebra
$\Li'.$ 
We illustrate this procedure 
by studying the classical examples 
of the $\Gf$-structures treated, e.g.
in \cite{Kob, Sternberg}.
In \cite{Ottazzi2011} the authors
 obtained this rank condition for contact
structures, when $\Li$ is the full linear group of $\gt_{-1}.$
\par
 When $\mt$ has kind $\muup{>}1,$ 
 the effect of the terms $\gt_{-\pq}$ with $\pq{\geq}2$ is of
 \textit{restricting the 
structure algebra $\Li$ 
to a smaller algebra $\Li'\! \! .$} 
The final criterion is obtained by 
applying Theorem~\ref{thm-t-4-6} to $\Li'\! \! .$
 In fact, there is a difference in the way the $\gt_{-\pq}$ summands
 contribute to $\Li'$ 
 between the cases $\pq{=}2$ and $\pq{>}2.$ \par
In \S\ref{sec5} we study 
the  maximal $\gl_{\K}(V)$-prolongation in the case $\muup{=}2.$ 
We show in Theorem~\ref{thm-tan-4-8} 
that the condition can be expressed in terms of the rank of
the Lie product, considered as 
an alternate bilinear form on $\Fb{\otimes}\gt_{-1},$ 
for the algebraic
closure $\Fb$ of the ground field $\K.$ 
We show by an example 
that this rank condition, that was known to be necessary
over $\K$ (see e.g. \cite{Kob}),  
becomes necessary \textit{and} sufficient 
when stated over~$\Fb.$  
As in \S\ref{sec5a}, 
there is an equivalent formulation in terms
of the rank a suitable matrix~$M_2(\Ki,\zq).$   
\par 
In \S\ref{sect7} we show that one can take into account the 
non zero summands $\gt_{-\pq}$ in $\mt,$ with $\pq\geq{3},$ 
by  adding to $M_2(\Ki,\zq)$ a  matrix $M_3(\Ki,\zq),$ 
which has rank $n{=}\dim\gt_{{-}1}$ when $\zq$ 
does not belong to a subspace $W(\Ki)$ of $\Fb{\otimes}V.$ 
Thus the finite dimensionality criterion of  
\S\ref{sec5} has to be checked on
a smaller subspace of~$\zq$'s.
\par
 In \S\ref{sect8},
 Theorem~\ref{thm-tan-6.1} collects the partial 
 results of the previous sections,
to state a finiteness criterion 
for the maximal $\Li$-prolongation 
of an $\mt$ of any finite kind $\muup$  
in terms of the rank of the matrix 
$(M_1(\Li,\zq),M_2(\Ki,\zq),M_3(\Ki,\zq))$ which is obtained
by putting together the contributions 
coming from $\Li$ (kind one), $\gt_{-2}$ (kind two) and
${\sum}_{\pq\leq{-3}}\gt_{\pq}$ (kind ${>}2$).

\subsubsection*{Notation} 
\begin{itemize}
 \item We shall indicate by $\K$ a field of characteristic zero and by $\Fb$ its algebraic closure.
 \item The acronym FGLA stands for \textit{fundamental graded Lie algebra}.
 \item The acronym EPFGLA stands for \textit{effective
 prolongation of a fundamental graded Lie algebra}.
 \item $\T(V)={\sum}_{\pq=0}^\infty\T^{\pq}(V)$ is the tensor algebra of the vector space $V.$ 
 \item $\Lambda(V)={\sum}_{\pq=0}^n\Lambda^{\pq}(V)$ is the Grassmann algebra of
 an $n$-dimensional vector space $V.$ 
 \item $\gl_{\K}(V)$ is the Lie algebra of $\K$-linear endomorphisms of the $\K$-linear space $V.$
 \item $\gl_n(\K)$ is the Lie algebra of $n{\times}n$ matrices with entries in the field $\K.$ 
\end{itemize}
\section{Fundamental graded Lie algebras and prolongations} \label{sec-fund}
\begin{defn}[see \cite{Tan70}] 
A $\Z$-graded Lie $\K$-algebra 
\begin{equation}\label{eq-ta-1-1} \tag{FGLA}
 \mt={\sum}_{\pq\geq{1}}\gt_{-\pq},
\end{equation}
is called \emph{fundamental} if $\gt_{-1}$
is 
finite dimensional 
and generates $\mt$
as a Lie $\K$-algebra. 
We call $\muup{=}\sup\{\pq\mid\gt_{{-}\pq}{\neq}\{0\}\}$
the \emph{kind} of $\mt.$ 
\end{defn}
At  variance with \cite{Tan70}, we do not require here that $\mt$ be finite dimensional.
The only FGLA with $\dim(\gt_{-1}){=}1$ is the trivial one-dimensional Lie algebra. We will
consider in the following FGLA's $\mt$ with $\dim(\mt){>}1.$ 
\par\smallskip
Let $V$ be a vector space of finite dimension
$n{\geq}2$ over $\K$ 
and  
$\ft(V)$ the free  Lie $\K$-algebra 
generated by $V$ 
(see e.g. \cite{n1998lie, Reu1993}).
It is an FGLA 
of infinite kind with the natural 
\mbox{$\Z${-}gradation} 
\begin{equation}
 \ft(V)={\sum}_{\pq=1}^\infty \ft_{-\pq}(V).
\end{equation}
obtained by setting $\ft_{-1}(V){=}V.$  
It
is characterised by the universal property: 
\begin{prop}\label{prop-ta-1-1}
To every Lie $\K$-algebra $\Li$  
and to every $\K$-linear map $\phiup{:}V{\to}\Li,$ there
is a unique Lie algebras homomorphism $\tilde{\phiup}:\ft(V){\to}\Li$ extending $\phiup.$
\qed
\end{prop}
By Proposition~\ref{prop-ta-1-1} we can 
consider any finitely generated Lie algebra as a quotient of 
a free Lie algebra. 
\par 
\begin{rmk}
For any integer $\muup\geq{2},$ the direct sum 
\begin{equation}
 \ft_{[\muup]}(V)={\sum}_{\pq{\geq}\muup}\ft_{-\pq}(V)
\end{equation}
is a proper $\Z$-graded ideal in $\ft(V)$ and hence the quotient 
\begin{equation}
 \ft_{(\muup)}(V)=\ft(V){/}\ft_{[\muup{+}1]}(V)
\end{equation}
is an FGLA of kind $\muup,$ that is called
the \emph{free} FGLA \emph{of kind $\muup$}
generated by $V$
(see e.g.~\cite{Warhurst2007}).
\end{rmk}

Let $\mt$ be an FGLA over $\K$ 
and set $V=\gt_{-1}.$ Since $\mt$ is generated by $V,$  there is 
a surjective  homomorphism  
\begin{equation}
 \piup:\ft(V) \longrightarrow\!\!\!\!\rightarrow \mt,
\end{equation}
which preserves the gradations. 
Its kernel $\Ki$ is a $\Z$-graded ideal
\begin{equation}\label{t-6-10eq}
 \Ki={\sum}_{\pq{\geq}2}\Ki_{\;-\pq}
\end{equation}
and $\mt$ 
has finite kind $\muup$ iff there is a smaller integer
$\muup{\geq}1$ such that $$\Ki\supseteq\ft_{[\muup{+}1]}(V).$$  
\begin{defn}
An ideal $\Ki$ of $\ft(V)$ is said to be
\emph{$\muup$-cofinite} if $\ft_{[\muup{+}1]}(V)\subseteq\Ki.$ 
\end{defn}
\begin{lem}
The elements of 
 a string 
 $(\Ki_{\;-2},\hdots,\Ki_{\;-\muup},\hdots)$ of linear subspaces 
 $\Ki_{\;-\pq}{\subseteq}\ft_{-\pt}(V)$
 are the homogeneous summands of a $\Z$-graded ideal \eqref{t-6-10eq}
 of $\ft(V)$ if and only if they satisfy the compatibility
 condition 
\begin{equation}
\vspace{-20pt}
 [\Ki_{\; -\pq},V]\subseteq\Ki_{\;-(\pq{+}1)},\;\;\forall \pq{\geq}2.
\end{equation}
\qed
\end{lem} 
Since $\K$ has characteristic $0,$ we can 
canonically identify $\ft(V)$
with a $\GL_{\K}(V)$-invariant subspace of $T(V).$  

\begin{prop}\label{prop-tan-1.4}
 Every FGLA $\mt$ 
 is isomorphic to a quotient 
 $\ft(V){/}\Ki,$ for a $\Z${-}graded ideal $\Ki$ of the form~\eqref{t-6-10eq}.
 Two FGLA's $\ft(V){/}\Ki$ and 
 $\ft(V){/}\Ki'$ 
 are isomorphic
 if and only if $\Ki$ and $\Ki'$ are $\GL_{\K}(V)$-congruent. \qed
\end{prop}

\begin{defn} A $\Z$-graded Lie algebra over $\K$ 
\begin{equation}\label{eq-ta-1-6}
 \gt={\sum}_{\pq\in\Z}\gt_{\pq}.
\end{equation} 
is said to be an 
\emph{effective prolongation of 
a fundamental graded Lie algebra}~if 
\begin{itemize}
 \item[($i$)] $\gt_{{<}0}={\sum}_{\pq{<}0}\gt_{\pq}$ is a FGLA;
 \item[($ii$)] $\gt_{{\geq}0}={\sum}_{\pq{\geq}0}\gt_{\pq}$ is \emph{effective}: this means that 
\begin{equation*}
 \{X\in\gt_{{\geq}0}\mid [X,\gt_{-1}]{=}\{0\}\}=\{0\}.
\end{equation*}
\end{itemize}
In this case, we say  for short 
that $\gt$ is an EPFGLA of $\gt_{{<}0}.$ \par
Let us fix a Lie subalgebra $\Li$ of $\gl_{\K}(\gt_{-1}),$ that we will call
the \emph{structure algebra}.
An EPFGLA \eqref{eq-ta-1-6} is said to be \emph{of type $\Li$}, or
an $\Li$-pro\-long\-a\-tion,  
if
for each $A\in\gt_0$ the map
$\vq\to[A,\vq],$ for $\vq\in\gt_{-1},$ is an element of $\Li$. 
 \end{defn} 
\begin{lem}\label{lem-ta-1-5}
 If \eqref{eq-ta-1-6} is an EPFGLA, then $\dim(\gt_{\pq}){<}\infty$ for all $\pq\in\Z.$ 
 \end{lem} 
\begin{proof}
Let $\gt={\sum}_{\pq\in\Z}\gt_{\pq}$ be an EPFGLA
and set $V=\gt_{-1}.$ For each $\xiup{\in}\gt,$ we define by recurrence 
\begin{equation*}\begin{cases}
 \xiup(\vq_0)=[\xiup,\vq_0], & \forall\vq_0\in{V},\\
 \xiup(\vq_0,\vq_1)=[[\xiup,\vq_0],\vq_1], &\forall\vq_0,\vq_1\in{V},\\
 \xiup(\vq_0,\hdots,\vq_\pq)=[\xiup(\vq_0,\hdots,\vq_{\pq{-}1}),\vq_{\pq}], &\forall\vq_0,\vq_1,\hdots,\vq_{\pq}\in{V}.
 \end{cases}
\end{equation*}
By the effectiveness assumption, for each
$\pq{\geq}0$ the map which associates to $\xiup{\in}\gt_{\pq}$ the $(\pq{+}1)$ multilinear map
$V^{\pq{+}1}\ni (\vq_0,\hdots,\vq_{\pq})\to\xiup(\vq_0,\hdots,\vq_{\pq})\in{V}$
is injective from $\gt_{\pq}$ to $V{\otimes}\T^{\pq{+}1}(V^*).$ 
\end{proof}
We already noted that, since  
$\K$ has characteristic zero, 
$\ft(V)$ can be canonically identified 
with a Lie subalgebra of the Lie algebra of  
the tensor algebra $\T(V)$; 
$\gl_{\K}(V)$ acts as an  algebra of zero degree derivations
on $\T(V),$  leaving $\ft(V)$ invariant. Denote by
$T_{\!{A}}$ the derivation of $\T(V)$ associated to
$A\in\gl_{\K}(V).$ With the Lie product defined by 
\begin{equation}
 [A,X]=T_{\!{A}}(X),\;\;\forall A\in\gl_{\K}(V),\;\forall X\in\ft(V),
\end{equation}
the direct sum 
\begin{equation}
 \tilde{\ft}(V)={\sum}_{\pq\geq{0}}\ft_{-\pq}(V),\;\;\text{with}\;\; \ft_0(V)=\gl_{\K}(V)
\end{equation}
is an EPFGLA 
of $\ft(V)$ and, for any Lie subalgebra $\Lt$ of
$\gl_{\K}(V),$ the sum $\Lt\oplus\ft(V)$ is a graded Lie subalgebra of $\tilde{\ft}(V).$ \par
By using \cite[Ch.2,\S{2},Prop.8]{n1998lie}, we obtain 
\begin{prop}\label{prop1.6}
 The maximal effective prolongation $\Ft(V)$ of $\ft(V)$ is 
\begin{equation}
 \Ft(V)={\sum}_{\pq=-\infty}^\infty \ft_{\pq}(V),
\;\;\text{where}\;\; 
 \ft_{\pq}(V)={V}\otimes\T^{\pq{+}1}(V^*), \;\;\text{for}\;\; \pq{\geq}0,
\end{equation}
is the space of $(\pq{+}1)$-linear $V$-valued maps on $V$ and
the Lie product is defined in such a way that, for $\xiup\in\ft_{p}(V),$ with $\pq{\geq}{0},$
we have 
\begin{equation}
\label{eq1.11}
[\xiup,\vq_0](\vq_1,\hdots,\vq_{\pq})=
 \xiup(\vq_0,\hdots,\vq_{\pq}),\;\;\forall\vq_0,\vq_1,\hdots,\vq_{\pq}\in{V}.
\end{equation}
\end{prop} 
\begin{proof} 
If $\Mi$ is a right $\ft(V)$-module, then we can define on $\Mi{\oplus}\ft(V)$ a Lie $\K$-algebra
structure by setting
\begin{equation}\label{eq1star}\tag{$*$}
 [(\mu,X),(\nu,Y)]=(\mu{\cdot}Y-\nu{\cdot}X,[X,Y]),\;\;\forall \mu,\nu\in\Mi,\;\;\forall{X},Y\in\ft(V).
\end{equation}
Let  $\xiup{:}V{\to}\Mi$ be any $\K$-linear map. By the universal property the $\K$-linear map 
$V{\ni}\vq{\to}(\xiup(\vq),\vq){\in}\Mi{\oplus}\ft(V)$
extends to a Lie $\K$-algebras homomorphism 
\begin{equation*}
\ft(V)\ni X\to (D_{\xiup}(X),X)\in\Mi{\oplus}\ft(V)\end{equation*} 
between $\ft(V)$ and $\Mi{\oplus}\ft(V).$ Because of \eqref{eq1star},
we have $D_{\xiup}{\in}\Der(\ft(V),\Mi).$ 
\par
For $\pq{\in}\Z,$  we define recursively
finite dimensional $\K$-vector spaces $\Der_{\pq}(V)$
and right $\Z$-graded $\ft(V)$-modules $\Der_{(\pq)}(V)={\sum}_{\qq{\leq}\pq}\Der_{\qq}(V)$ 
 by setting 
\begin{equation*}
\begin{cases}
 \Der_{\pq}(V)=\ft_{\pq}(V),\; \Der_{(\pq)}(V)={\sum}_{\qq{\leq}\pq}\ft_{\qq}(V), &\text{for $\pq{<}0,$}\\[6pt]
 \left.
 \begin{aligned}
 &
 \Der_{\pq}(V)=\{D{\in}\Der(\ft(V),\Der_{(\pq{-}1)}(V)){\mid} \, D(\vq){\in}\Der_{\pq{-}1}(V),\;\forall\vq{\in}V\},
\\
&
 \Der_{(\pq)}(V)= \Der_{(\pq{-}1)}(V)\oplus\Der_{\pq}(V),
 \end{aligned}\right\}
  &\text{for $\pq{\geq}0,$}
\end{cases}
\end{equation*}
By the general argument at the beginning,    
$\Der_{\pq}(V){\simeq}V{\otimes}T^{{\pq}{+}1}(V^*)$ for $\pq{\geq}{0}$ 
and  \eqref{eq1.11} holds.\par
The last step is setting $\Der_{\pq}(V){=}\ft_{\pq}(V)$ and extending the Lie $\K$-algebra structure
of $\ft(V)$ to $\Ft(V).$ In doing that, 
we can restrain to homogeneous terms and argue by recurrence on the sums
of their degrees. Let $\xiup_{i}$ denote an element of $\ft_{\pq_i}(V)$ and consider a product
$[\xiup_{1},\xiup_{2}].$ It is well defined when $\min\{\pq_1,\pq_2\}{<}0.$ If we assume it
has been already defined for $\pq_1{+}\pq_2{\leq}\kq$ for some $\kq{\geq}0,$ we note that 
$[\xiup_{i},\vq]{\in}\ft_{\pq_i{-}1}$ and hence 
\begin{equation}\label{eq1sstar}\tag{$**$}
[ [\xiup_{1},\xiup_{2}],\vq]=[ [\xiup_{1},\vq],\xiup_{2}]
+[ \xiup_{1},[\xiup_{2},\vq]]
\end{equation}
is a well defined $\K$-linear map between $\vq{\in}V$ and $\Der_{\pq_1{+}\pq_2{-}1}(V),$ 
defining an element $[\xiup_{\pq_1},\xiup_{\pq_2}]$ of $\Der_{\pq_1{+}\pq_2}(V).$ 
This yields a product on $\Ft(V).$ To show that it satisfies 
the Jacobi identity 
\begin{equation}\label{eq1ssstar} \tag{$***$}
 [[\xiup_{1},\xiup_{2}],\xiup_{3}]+ [[\xiup_{2},\xiup_{3}],\xiup_{1}]
 + [[\xiup_{3},\xiup_{1}],\xiup_{2}]=0,
\end{equation}
we note that \eqref{eq1ssstar}
follows from \eqref{eq1sstar} when $\min\{\pq_1,\pq_2,\pq_3\}{<}0$ and can be proved by
recurrence on $\kq{=}\pq_1{+}\pq_2{+}\pq_3$ when $\min\{\pq_1,\pq_2,\pq_3\}{\geq}0,$ 
because in this case 
an element $\xiup$ of $\ft_{\kq}(V)$ is zero iff $[\xiup,\vq]{=}0$ for all
$\vq{\in}V.$ 
\end{proof}
\par\medskip

\begin{prop}
 Let $\Li$ be a Lie subalgebra of $\gl_{\K}(V).$ Then 
\begin{equation}
 \label{eq-ta-1-12b} \left\{\begin{aligned}
 \Ft(V,\Li)&=\ft(V)\oplus\Li\oplus{\sum}_{\pq{\geq}1}\ft_{\pq}(V,\Li),\;\;\text{with}\;\; \\
& \ft_{\pq}(V,\Li)=
 \{\xiup\in\ft_{\pq}(V)\mid \xiup(\vq_1,\hdots,\vq_{\pq})\in\Li,\;\forall \vq_1,\hdots,\vq_{\pq}
 \in{V}\}\end{aligned}\right.
\end{equation}
is the maximal EPFGLA prolongation of type $\Li$ of
$\ft(V).$  \qed
\end{prop}

\par\medskip

Fix a $\Z$-graded ideal $\Ki$ of $\ft(V),$ 
contained in $\ft_{[2]}(V),$ and denote 
by 
\begin{equation} \label{eq-1.13}
\mt(\Ki){=}{\sum}_{\pq{\geq}1}\gt_{-\pq}(\Ki)
\end{equation}
the FGLA defined by the quotient
$\ft(V){/}\Ki.$ 
Let $\Li$ be 
a Lie subalgebra  of $\gl_{\K}(V).$ 
Since
$\Ki$ is a Lie subalgebra of $\Ft(V,\Li),$ 
we can 
associate 
to the pair $(\Ki,\Li)$ 
the normalizer 
\begin{equation}
\Nt(\Ki,\Li)=\{\xiup\in\Ft(V,\Li)\mid [\xiup,\Ki]\subseteq\Ki\}
\end{equation}
of $\Ki$ in 
$\Ft(V,\Li).$ It is 
the largest Lie subalgebra of $\Ft(V,\Li)$ containing 
$\Ki$ as an ideal and is $\Z$-graded.
The quotient 
\begin{equation} \label{eq-1.14}
 \gt(\Ki,\Li)=\Nt(\Ki,\Li)/\Ki
\end{equation}
has a natural $\Z$-grading for which  $\gt_{{<}0}(\Ki,\Li){=}\mt(\Ki)$ and 
the natural projection 
$\Nt(\Ki,\Li){\rightarrow\!\!\!\!\rightarrow} 
\gt(\Ki,\Li)$ is a $\Z$-graded epimorphism of Lie algebras.
\begin{thm}\label{thm-tan-1-8}
 The commutative diagram 
 \begin{equation*}
 \xymatrix@1{ 0 \ar[r] &\ft(V)\ar[r]\ar@{->>}[d] & \Nt(\Ki,\Li)\ar@{->>}[d] \\
 0\ar[r] & \mt(\Ki) \ar[r] & \gt(\Ki,\Li)
 }
\end{equation*}
defines on $\gt(\Ki,\Li)$ the structure of an EPFGLA of type $\Li$ of $\mt(\Ki)$
and $\gt(\Ki,\Li)$ is, modulo isomorphisms, its maximal 
EPFGLA 
of type $\Li.$ \qed
\end{thm}

We have 
\begin{equation}\label{eq-ta-1-12a} \left\{ 
\begin{aligned}
 \gt(\Ki,\Li)&={\sum}_{\pq\in\Z}\gt_{\pq}(\Ki,\Li), \;\;\text{with
 $\gt_{\pq}(\Ki,\Li)=\gt_{\pq}(\Ki)$ for $\pq{<}0,$ and}\\
 \gt_{\pq}(\Ki,\Li)&=\{\xiup\in\ft_{\pq}(V,\Li))
 \mid [\xiup,\Ki]\subseteq\Ki\},\;\;\text{for $\pq{\geq}0.$}  \end{aligned} \right.
\end{equation}
\par\medskip
{\small
\begin{rmk}
This theorem was proved in \cite[\S{5}]{Tan70} 
under the additional assumption that $\Ki$ is cofinite.
The 
summands $\gt_{\pq}(\Ki,\Li)$ in \eqref{eq-ta-1-12a}
were 
 recursively defined there by setting \vspace{-15pt}
\begin{equation}
 \label{eq-ta-1-12} 
\begin{cases}
\gt_{\pq}(\Ki,\Lt)=\gt_{\pq}(\Ki), &\text{for $\pq{<}0,$}\\
\left.
\begin{aligned}
& \gt_{{<}\pq}(\Ki,\Lt)={\sum}_{\qq{<}\pq}\gt_{\qq}(\Ki,\Lt),  
 \\
 &\gt_{\pq}(\Ki,\Lt)=\Der_{\!\pq}(\mt,\gt_{{<}\pq}(\Ki,\Lt)), \end{aligned} \right\}&\text{for $\pq{\geq}0.$}
\end{cases}
\end{equation}
Each  $\gt_{{<}\pq}(\Ki,\Lt)$ is a right $\mt(\Ki)$-module and 
$\Der_{\!\pq}(\mt,\gt_{{<}\pq}(\Ki,\Lt))$ is the space of the 
$\gt_{{<}\pq}$-valued 
degree $\pq$-derivations of $\mt.$ 
One can check that this definition is equivalent to the one 
we gave above (cf. the proof of Proposition~\ref{prop1.6}).
\end{rmk}}

\section{Tanaka's finiteness criterion}\label{sect2}
\subsection{Left and right $V$-modules} \label{sec-t-1}
Let $V$ be an $n$-dimensional $\K${-}vector space, 
that we consider as an abelian Lie algebra. 
Its universal enveloping algebra is 
the  graded algebra
 \begin{equation}
\Az={\sum}_{\pq=0}^\infty\Symm_\pq(V)
\end{equation}
of symmetric elements of its tensor algebra. Since $\K$ has characteristic zero, 
$\Az$ is the ring of polynomials of $V$ with coefficients in $\K.$
\par 
A right action of $V$ on a $\K$-vector space $E$ is a bilinear map 
\begin{gather}\label{eqt1.1}
E\times{V}\ni (\eq,\vq)\longrightarrow \eq{\cdot}\vq\in{E},
\\
\notag
\text{with}\qquad
(\eq{\cdot}\vq_1){\cdot}\vq_2=(\eq{\cdot}\vq_2){\cdot}\vq_1,\;\;
\forall\eq\in{E},\;\;\forall \vq_1,\vq_2\in{V}.\qquad
\end{gather}
The map \eqref{eqt1.1} extends to a right action 
\begin{equation}
E\times\Az\ni (\eq,\aq)\longrightarrow\eq{\cdot}\aq\in{E}.
\end{equation}
If $E^*$ is a $\K$-vector space which is in duality with $E$
by a pairing \begin{equation}
E\times{E}^*\ni (\eq,\etaup)\longrightarrow\langle\eq\mid\etaup\rangle\in\K,
\end{equation} 
a \textit{left} action 
\begin{equation}\label{eq-ta-2.5a}
V\times{E}^*\ni(\vq,\etaup)\longrightarrow\vq{\cdot}\etaup\in{E}^*
\end{equation}
of $V$ on $E^*$ is \emph{dual}
of  \eqref{eqt1.1} if
\begin{equation}
\langle\eq\,|\,\vq{\cdot}\etaup\rangle=\langle\eq{\cdot}\vq\,
|\,\etaup
\rangle,\;\;\;\forall\vq{\in}V,\;\forall\eq{\in}E,\;\forall\etaup
{\in}E^*.
\end{equation}
Clearly a dual left action of $V$ extends to a left action
of $\Az.$ 

\begin{defn}
A \emph{$\Z$-gradation} of the right
$V$-module $E$ is a direct sum decomposition 
\begin{equation}\label{eq.t.1.2}
E={\sum}_{\pq=-\infty}^\infty E_\pq\;\;\; \text{with}\;\; 
E_\pq{\cdot}V\subseteq{E}_{\pq-1}\;\text{for}\; \pq\in\Z.
\end{equation}
\par 
We say that $E$ satisfies  
condition 
\eqref{cndC}
if \begin{equation}\label{cndC} \tag{$C$}
\begin{cases}
\dim(E_p)<+\infty,\;\forall p\in\Z,\quad \\
\exists \pq_0\in\Z\;\;{s.t.}\;\; E_{\pq}=\{0\} \;\;\text{for}\;\; 
\pq{<}\pq_0,\\
\eq\in{\sum}_{h\geq{0}}E_p \;\;\text{and}\;\; \eq{\cdot}V=\{0\} 
\Longrightarrow \eq=0.
\end{cases}
\end{equation} 
\end{defn}
\par\smallskip 
We call 
the $\Z$-graded vector space
\begin{equation}
E^*={\sum}_{h=-\infty}^\infty E^*_h, \;\;
\text{with $E^*_h=$ dual space of $E_{-h},$ for all $h\in\Z.$}
\end{equation}
the \emph{graded dual} of $E.$ 
When $E$ is a graded right $V$-module, its graded dual
$E^*$ is a 
graded left $V$-module under the action 
\eqref{eq-ta-2.5a} which is described on the homogeneous
elements by
\begin{equation}
\langle\eq\mid \vq{\cdot}\etaup\rangle=
\langle\eq{\cdot}\vq\mid\etaup\rangle
\;\;
\forall \eq\in{E}_{1-h},\;\forall\etaup\in{E}^*_h,\;\forall\vq\in{V}.
\end{equation}

\begin{lem} The following are equivalent \begin{itemize}
\item[(a)] $E$ satisfies condition \eqref{cndC};
\item[(b)] $E^*$ satisfies condition
\begin{equation}\tag{$C^*$} \label{cndstar}
\begin{cases}
\dim(E_\qq^*)<+\infty,\;\forall \qq\in\Z,\quad \\
\exists\, \qq_{\,0}\in\Z\;\;{s.t}\;\; E^*_{\qq}=0,\;\forall \qq>\qq_{\,0},\\
{\sum}_{\qq\geq{1}}E^*_{\qq}\;\;\text{generates}\; E^*\;
\text{as a left $\Az$-module.}
\end{cases}
\end{equation}
\end{itemize}
\end{lem} 
\begin{proof}
We have $E_{\pq}=\{0\}$ iff $E^*_{-\pq}=\{0\}.$  \par 
Given any basis 
$\vq_1,\hdots,\vq_n$ of $V,$ the
maps \begin{gather*}
 E_{\pq}\ni\eq\to(\eq{\cdot}\vq_1,\hdots,\eq{\cdot}\vq_n)\in({E}_{\pq-1})^n,
\\  
({E}^*_{1-\pq})^n\ni(\etaup_1,\hdots,\etaup_n)
\to \vq_1{\cdot}\etaup_1{+}\cdots{+}\vq_n{\cdot}\etaup_n\in{E}^*_{-\pq}
\end{gather*}
are  dual of each other and all vector spaces involved are
finite dimensional. 
Hence
the first one is injective if and only if the second one is surjective. 
These remarks yield the statement.
\end{proof}
Let $(V)$ denote the maximal ideal of $\Az,$ 
consisting of polynomials
vanishing at $0.$ For the notions of commutative algebra that
will be used below 
we refer to \cite[Ch.IV]{Bour89}
\begin{prop} \label{propt.1.2}
For an $E$ satisfying condition \eqref{cndC} the following are equivalent: 
\begin{itemize}
\item[$(i)$] $E$ is finite dimensional;
\item[$(ii)$] $E^*$ is finite dimensional;
\item[$(iii)$] $E^*$ is $(V)$-coprimary, i.e. all $\vq{\in}V$ are nilpotent
on $E^*.$
\end{itemize}
\end{prop} \begin{proof}
Clearly $(i)\Leftrightarrow(ii).$ Also $(ii)\Rightarrow(iii)$ 
is clear. Indeed, if
$E^*$ is finite dimensional, all $\vq\in{V},$ lowering the degree by
one unit, are nilpotent. Vice versa, since $E^*$ is finitely generated, 
if the set $\mathpzc{Ass}(E^*)$ of its associated primes is 
$\{(V)\},$
then all $E^*_\qq$ are zero for
$\qq{<}\qq_{\,0}$ for some $\qq_{\,0}\in\Z.$
\end{proof}
\begin{cor}
A necessary and sufficient condition for a right $V$-module $E$
satisfying condition \eqref{cndC} to be infinite dimensional is that
there is $\vq\in{V}$ such that $E{\cdot}\vq{=}E.$ 
\end{cor} \begin{proof}
By Proposition~\ref{propt.1.2} a necessary and sufficient condition
for $E$ to be infinite dimensional is that $E^*$ is not 
$(V)$-coprimary. If $\pri_1,\hdots,\pri_m$ are the associated primes
of $E^*,$ since $V{\cap}\,\pri_j=V_{\!j}$ is, 
for each $j,$ a proper vector subspace
of $V,$ it suffices to choose a $\vq$ which does not belong to 
  $V_1{\cup}\cdots{\cup}V_m.$ \par 
In fact, for such a $\vq,$ all maps $E^*_{\qq}{\ni}\etaup{\to}
\vq{\cdot}\etaup\in{E}^*_{\qq{-}1}$ are injective. By duality,
all maps $E_{\pq}\ni{\eq}{\to}\eq{\cdot}\vq\in{E}_{\pq{-}1}$
are surjective.  This gives $E{\cdot}\vq{=}E.$\par 
Vice versa, if right multiplication by $\vq$ on $E$
in surjective, by duality left multiplication by $\vq$
is injective on $E^*$ and hence $E^*$ is not 
$(V)$-coprimary.
\end{proof}
It will be useful that the thesis of this corollary be verified
by a $\vq$ which has a special form. To this aim we prove the
following lemma.
\begin{lem} \label{lem-t.1.4}
Let $U,W$ be finite dimensional vector spaces over $\K$ 
and $\omegaup:U\times{W}\to{V}$ a bilinear map such that 
$\{\omegaup(\uq,\wq)\mid \uq\in{U},\,\wq\in{W}\}$ 
spans $V.$ 
Then the image of $\omegaup$ is not contained in any finite union
of proper linear subspaces of $V.$ 
\end{lem} \begin{proof}
It suffices to show that $\omegaup(U\times{W})$ is not contained in
a finite union of hyperplanes. Let $\xiup_1,\hdots,\xiup_m\in{V}^*$
and assume that $\omegaup(U\times{W})$ is contained in
$\ker(\xiup_1)\cup\cdots\cup\ker(\xiup_m).$ The condition that
$\omegaup(U\times{W})$ spans $V$ implies that each quadric
$Q_j=\{(\uq,\wq)\in{U}\times{W}\mid \xiup_j(\omegaup(\uq,\wq))=0\}$
is properly contained in $U{\times}W.$ Hence
$Q_1{\cup}\cdots{\cup}Q_m\subsetneqq{U}\times{W}$ and the $\omegaup$-image 
of each pair which is not contained in $Q_1{\cup}\cdots{\cup}Q_m$
does not belong to $\ker(\xiup_1){\cup}\cdots{\cup}\ker(\xiup_m).$
\end{proof}
\begin{prop}\label{prop-t.1.4}
 Let $E$ be a right $V$-module satisfying condition \eqref{cndC}.
 If $E$ is infinite dimensional, then we can find 
linearly independent vectors 
$\vq_1,\hdots,\vq_m$ in $V$
 such that, with 
\begin{equation}
 E^{(h)}=\{\eq\in{E}\mid \eq{\cdot}\vq_i=0,\;\forall i{\leq}h\},\;\; \text{for}\;\; 0{\leq}h{\leq}m,
\end{equation}
the following conditions are satisfied:
\begin{itemize}
 \item[$(i)$] either $m{=}n$ and $\vq_1,\hdots,\vq_n$ is a basis of $V,$ or $0{<}m{<}n$ and 
 $E^{(m)}$ is finite dimensional;
 \item[$(ii)$]  the maps
 $T_{h+1}:E^{(h)}{\ni}\eq{\to}\eq{\cdot}\vq_{h+1}{\in}{E}^{(h)}$ 
 are surjective for \; $0{\leq}{h}{\leq}{m{-}1}.$ 
\end{itemize}
\end{prop} 
\begin{proof} Note that $E^{(0)}=E.$ 
We know from
Proposition~\ref{propt.1.2}  that, if $\dim(E)=\infty,$ then 
$E^*$ is not $(V)$-coprimary. Hence we can find
 $\vq_1\in{V}$ such that the left multiplication by $\vq_1$ is injective on $E^*.$  
 By duality this means that the map $T_1:E\ni\eq\to\eq{\cdot}\vq_1\in{E}$ is surjective.
 \par 
 Let
 $V_1$ be a hyperplane in $V$ which does not contain $\vq_1.$ We have a natural
 duality pairing between $E^{(1)}$ and the quotient
$E^*{/}\vq_1{\cdot}E^*,$  that we can consider as
a left $V_1$-module. In case $E^{(1)}$ is infinite dimensional,
$(V_1)$ is not an associated prime of $E^*{/}\vq_1{\cdot}E^*$ and hence we can find
$\vq_2{\in}V_1$ which is not a zero divisor in $E^*{/}\vq_1{\cdot}E^*.$ By duality this yields
a surjective $$T_{2}:E^{(1)}\ni\eq\to\eq{\cdot}\vq_2\in{E}^{(1)}.$$ \par
The recursive argument is now clear and we get the statement after a finite number of steps. 
\end{proof}
\begin{rmk}
In the hypothesis of Lemma~\ref{lem-t.1.4}, the vectors $\vq_1,\hdots,\vq_m$ in the statement
of Proposition~\ref{prop-t.1.4} can be chosen of the form 
$\vq_i=\omegaup(\uq_i,\wq_i),$ with $\uq_i\in{U}$ and $\wq_i\in{W},$ 
for $1{\leq}i{\leq}m.$ 
\end{rmk}
\subsection{Right $\mathfrak{m}$-modules} \label{sec2}
Let $\mt={\sum}_{\pq=1}^\muup\gt_{-\pq}$ be an FGLA of finite kind $\muup{\geq}1.$ 
\par 
A $\Z$-graded right $\mt$-module
is a $\Z$-graded $\K$-vector space 
on which a $\K$-bilinear map 
\begin{equation*}
 E\times\mt\ni (\eq,X)\longrightarrow \eq{\cdot}X\in{E},
\end{equation*}
is defined, with the properties:
\begin{equation*} 
\begin{cases}
( \eq{\cdot}X){\cdot}{Y}-( \eq{\cdot}Y){\cdot}{X}=\eq{\cdot}[X,Y],\;\;\forall \eq\in{E},\;\forall X,Y\in\mt,\\
E_{\qq}{\cdot}\gt_{-\pq}\subseteq{E}_{\qq-\pq},\;\;\forall \qq\in\Z,\;\forall 1{\leq}\pq{\leq}\muup.
\end{cases}
 \end{equation*}
\begin{defn}
 We say that the graded right $\mt$-module $E$ satisfies condition \eqref{cndCi} if 
\begin{equation} \tag{$C(\mt)$}
 \label{cndCi} 
\begin{cases}
\dim(E_{\pq})<\infty, \;\;\forall \pq\in\Z,\\
 \exists\,\pq_{0}\in\Z\;\text{s.t.}\; E_{\pq}=0,\;\;
 \forall \pq<-\pq_0,\\
 \eq\in{\sum}_{\pq\geq{0}}E_{\pq}\;\;\text{and}\;\; \eq{\cdot}\gt_{-1}=\{0\}\;\Longrightarrow\; \eq=0.
\end{cases}
\end{equation}
 \end{defn}
The following lemma is similar to \cite[Lemma 11.4]{Tan70}.
\begin{lem} \label{lem-t.2.3}
Assume that \begin{itemize}
\item 
$\mt$ has kind 
$\muup{\geq}2;$  
\item 
 $E$ is a graded right $\mt$-module satisfying condition \eqref{cndCi}; 
 \item  there
 are $X\in\gt_{-1},$ $Y\in\gt_{1-\muup}$ and $\pq_0\in\Z$ such that 
\begin{equation}
 \psiup_{\pq}:E_{\pq}\ni \eq\longrightarrow \eq{\cdot}[X,Y]\in{E}_{p-\muup}
\end{equation}
is an isomorphism for $\pq{\geq}\pq_0.$
\end{itemize}
Then $E_{\pq}=\{0\}$ for $\pq{\geq}{\pq}_0.$
\end{lem} 
\begin{proof}
 We consider the right multiplication by $Z=[X,Y]$ 
 as a linear map $R_Z$ on $E$. 
Note that $E_{<{\qq_{\,0}}}={\sum}_{\pq<\qq_{\,0}}E_{\pq}$ is a right
$\mt$-submodule of $E$. Thus we can consider the \textit{truncation}
${\sum}_{\pq\geq\qq_{\,0}}E_{\pq}$ as the quotient right $\mt$-module
$E{/}E_{{<}{\qq_{\,0}}}.$\par 
By substituting to $E$ its truncation, if needed, 
 we can as well assume that
  $E={\sum}_{\pq\geq\pq_0-\muup}E_{\pq}.$ Then
 $R_Z$ has a right inverse $\Psi:E\to{E}$ 
 such that $\Psi{\circ}R_Z$ restricts to the identity
 on each $E_p$ with $p{\geq}p_0,$ 
i.e.  $\Psi$ is also a left inverse of $R_Z$ on $E_{\pq}$ for $\pq{\geq}\pq_0.$ 
 Denote by $R_X$ and $R_Y$ the linear maps on $E$ defined 
 by the right multiplication by $X$ and $Y,$ respectively. We claim that 
\begin{equation*}
 \Psi{\circ}R_X=R_X{\circ}\Psi
 \quad\text{and}\quad  \Psi{\circ}R_Y=R_Y{\circ}\Psi
 \;\;
 \text{on} \;\;
 E_{\pq},\;\;
 \forall \pq{\geq}\pq_0{-}1.
\end{equation*}
Indeed, $\Psi{\circ}R_X(E_{\pq})\cup{R}_X{\circ}\Psi(E_{\pq})\subseteq{E}_{\pq+\muup-1}$
and $\Psi{\circ}R_Y(E_{\pq})\cup{R}_Y{\circ}\Psi(E_{\pq})\subseteq{E}_{\pq+1},$
so that, since both $\pq{+}1$ and $\pq{+}\muup{-}1$ are ${\geq}\pq_0$ when $\pq\geq\pq_0{-}1,$ 
and $\Psi$ is a right inverse of $R_Z$ in this range, these equalities are equivalent to 
\begin{equation*}
 R_X=R_{Z}{\circ}R_X{\circ}\Psi\quad\text{and}\quad R_Y=R_{Z}{\circ}R_Y{\circ}\Psi,
\end{equation*}
and thus are verified because $R_Z$ commutes 
with $R_X$ and $R_Y.$  
Fix $\pq{\geq}\pq_0{-1}$
and consider the finite dimensional $\K$-vector space $W={\sum}_{h=1}^{\muup}E_{\pq+h}.$ 
We define two endomorphisms $T_X,T_Y$ on $W$ by setting 
\begin{equation*}\begin{aligned}
 T_X(\eq)&= 
\begin{cases}
 R_X{\circ}\Psi(\eq)\in{E}_{\pq+\muup}, &\text{if $\eq\in{E}_{\pq+1},$}\\
 R_X(\eq)\in{E}_{h-1}, &\text{if $\eq\in{E}_{h},$ with $\pq{+}1{<}h\leq\muup,$}
\end{cases}\\
T_Y(\eq)&= 
\begin{cases}
 R_Y{\circ}\Psi(\eq)\in{E}_{h+1}, &\text{if $\eq\in{E}_h$ with $\pq{+}1{\leq}h<\pq{+}\muup,$}\\
 R_Y(\eq)\in{E}_{\pq+1} & \text{if $\eq\in{E}_{\pq+{\muup}} .$}
\end{cases}
\end{aligned}
\end{equation*}
One easily checks that 
\begin{equation*}
 T_X{\circ}T_Y-T_Y{\circ}T_X=(R_X{\circ}{R}_Y-R_Y{\circ}R_X){\circ}\Psi=\Id_W\;\;\;\text{on $W$}
\end{equation*}
and hence 
\begin{equation*}
 \dim(W)=\trac(\Id_W)=\trac( T_X{\circ}T_Y-T_Y{\circ}T_X)=0.
\end{equation*}
This proves the lemma. 
\end{proof}
Let 
$E={\sum}_{h=-\pq_0}^\infty E_{\pq}$ be a right $\mt$-module, satisfying condition~\eqref{cndCi}.
Set \begin{equation}\label{equ-ta-3.2}
\nt={\sum}_{h\geq{2}}\gt_{-h}\end{equation}
 and
\begin{equation} 
 \N(E)={\sum}_{h=-\pq_0}^\infty\N_{\pq}(E),\;\;\text{with}\;\; \N_{\pq}(E)=\{\eq\in{E}_{\pq}\mid \eq{\cdot}{\nt}=\{0\}\}.
\end{equation}
\begin{thm}\label{thm2.2}
 Let $E$ be a $\Z$-graded right $\Z$-module satisfying condition \eqref{cndCi}. Then $E$ is finite dimensional
 if and only if $\N(E)$ is finite dimensional.
\end{thm} 
\begin{proof}
 We argue by recurrence on the kind $\muup$ of $\mt.$ In fact, when $\muup=1,$ we have $\nt=\{0\}$
 and hence $\N(E)=E$ and the statement is trivially true. Moreover, since $\N(E)\subseteq{E},$ we only
 need to show that $N(E)$ is infinite dimensional when $E$ is infinite dimensional. \par
 Assume that $\muup{>}1.$ 
 The subspace $\gt_{-\muup}$ is an ideal of $\mt$ and hence $\mt'{=}\mt{/}\gt_{-\muup}$ is an 
 FGLA 
 of kind $\muup{-}1.$ Then
\begin{equation}
 F={\sum}_{\pq=-\pq_0}^\infty{F}_{\pq},\;\;\text{with}\;\; F_{\pq}=\{\eq\in{E}_{\pq}\mid \eq{\cdot}\gt_{{-}\muup}=\{0\}\}
\end{equation}
can be viewed 
as a $\Z$-graded $\mt'$-module which satisfies condition $C(\mt')$. \par
Since $\N(F){=}\N(E),$  
by our recursive assumption $F$ and $\N(E)$ are either both finite, or both infinite dimensional.\par
Hence it
will suffice to prove that, if $E$ is infinite dimensional, also
$F$ is infinite dimensional.  
The $\Z$-grading 
\begin{equation}
\M(E)={\sum}_{\pq\in\Z}\M_{\pq}(E),\;\;\text{where}\;\; \M_{\pq}(E)={\sum}_{j=0}^{\muup-1}E_{j+\pq{\cdot}\muup}.
\end{equation}
defines on $E$ the structure of 
a $\Z$-graded left $V$-module,  
 for $V=\gt_{-\muup}.$ 
 Since, by assumption,
  it is infinite dimensional, 
by Proposition~\ref{prop-t.1.4} 
and Lemma~\ref{lem-t.1.4} we can find $X_1,\hdots,X_m\in\gt_{-1}$
and $Y_1,\hdots,Y_m\in\gt_{1-\muup}$ such that 
$$Z_1{=}[X_1,Y_1],\hdots,Z_m{=}[X_m,Y_m]$$ 
are linearly independent and 
have the properties: 
\begin{itemize}
 \item[$(i)$] either $m{<}\dim_{\K}(\gt_{-\muup})$ and
 $E^{(m)}=\{\eq\in{E}\mid \eq{\cdot}Z_i=0,\;\forall i=1,\hdots,m\}$ is finite dimensional, or
 $m{=}n$ and $Z_1,\hdots,Z_n$ is a basis of $\gt_{-\muup}$; 
 \item[$(ii)$] with $E^{(h)}=\{\eq\in{E}\mid \eq{\cdot}Z_i=0,\;\forall i{\leq}h\},$ the maps 
\begin{equation*}
 T_{h+1}:E^{(h)} \ni\eq \longrightarrow \eq{\cdot}Z_{h+1}\in{E}^{(h)}
\end{equation*}
are surjective for $0{\leq}h{<}m.$ 
\end{itemize}
For all $0{\leq}h{<}m$ we obtain exact sequences 
\begin{equation} \label{eq-t.3.4}
\begin{CD}
 0 @>>> E^{(h+1)} @>>> E^{(h)} @>{{T_{h+1}}}>> E^{(h)} @>>> 0.
\end{CD}
\end{equation}
Note that $E^{(0)}=E$ and that $E^{(n)}=F$ when $m{=}n.$ 
We want to prove that ${m}=n$ and that $F$ is infinite dimensional. We argue by contradiction.
If our claim is false, then $E^{(m)}$ is finite dimensional and, from the exact sequence \eqref{eq-t.3.4},
we obtain that
$E^{(m-1)}_{\pq}\ni\eq\to\eq{\cdot}Z_{m}\in{E}^{(m-1)}_{\pq-\muup}$ is an isomorphism for all $\pq\geq\pq_m$
for some integer $\pq_m,$ which, 
by Lemma~\ref{lem-t.2.3}, implies that $E^{(m-1)}$ is finite dimensional.
Using again the exact sequence \eqref{eq-t.3.4} for $h{=}m{-}2$ and Lemma~\ref{lem-t.2.3} we obtain that
also 
$\dim_{\R}(E^{(m-2)})<\infty$ and, repeating the argument, we end up getting that $E^{(0)}=E$ is
finite dimensional. This yields a contradiction, proving that $F$,
and thus also $\N(E),$ must be infinite dimensional
when $E$ is infinite dimensional. 
\end{proof}

\subsection{Reduction to  first kind } 
Let us get back to the prolongations defined in \S\ref{sec-fund}.
We will use Theorem~\ref{thm2.2} to show that the finiteness of
the maximal effective $\Li$-pro\-long\-a\-tion of an FGLA of  
finite kind $\muup{\geq}2$ is equivalent to that of
the $\Lt'$-prolongation of an 
FGLA of the first kind  for a suitable~$\Lt'{\subseteq}\Lt.$ 
\par 
Let $\mt={\sum}_{\pq=1}^\muup\gt_{-\pq}$ be an FGLA of finite 
kind $\muup.$ Set $V{=}\gt_{{-}1}.$ 
We showed in~\S\ref{sec-fund} that
$\mt\simeq\ft(V){/}\Ki$ for a $\Z$-graded 
ideal $\Ki$ of $\ft(V),$
contained in $\ft_{[2]}(V).$ 
  As in \eqref{equ-ta-3.2}, we denote by $\nt$ the ideal
${\sum}_{\pq{\leq}{-2}}\gt_{\pq}$
of $\mt.$ 
Fix 
a Lie subalgebra  $\Li$ of $\gl_{\K}(\gt_{-1})$ and denote by $\gt(\Ki,\Li)$
the maximal EPFGLA  of type $\Li$ of 
$\mt,$ that was characterised 
in Theorem~\ref{thm-tan-1-8}.
We have the following finiteness criterion:
\begin{thm}\label{thm-t-4-2}
The maximal effective $\Li$-prolongation
 $\gt(\Ki,\Li)$  of
$\mt(\Ki)$ is finite dimensional
 if, and only if, 
\begin{equation}\label{equ-ta-4.2}
 \N(\gt(\Ki,\Li))=\{\xiup\in\gt(\Ki,\Li)\mid [\xiup,\nt]\}=\{0\}\}
\end{equation}
is finite dimensional. 
\end{thm} 
\begin{proof}
 The statement follows by applying 
 Theorem~\ref{thm2.2} to $\gt(\Ki,\Li),$ considered
 as a right $\mt$-module.
\end{proof}
\par\smallskip
Set
\begin{gather}
  \label{eq-t-4-11a}
  \at=(\mt{/}\nt)\oplus{\sum}_{\pg\geq{0}}\at_{\pq},
  \;\;\text{with}\;\;
 \at_\pq=\{\xiup\in\gt_\pq(\Ki,\Li)
 \mid [\xiup,\nt]=\{0\}\}. 
 \end{gather}
 We note that $\mt\oplus{\sum}_{\pq\geq{0}}\at_{\pq}$ 
 is a Lie subalgebra of $\gt(\Ki,\Li),$ which contains  
 $\nt$ as an ideal. 
 There is a natural isomorphism 
 of $\at$ with $(\mt\oplus{\sum}_{\pq\geq{0}}\at_{\pq}){/}\nt,$
which defines its Lie algebra structure.   
Set $\at_{-1}=\mt{/}\nt\simeq{V}.$ 
\begin{lem} \label{lem-tan-4-2}
The summand $\at_0$ in \eqref{equ-ta-4.2} is given by
\begin{equation} 
 \at_0=\at_0(\Ki,\Li)=\{A\in\Li\mid T_A(\ft_{[2]}(V))\subseteq\Ki\}.
\end{equation}
\end{lem} 
\begin{proof}
 The statement follows because $\nt=\ft_{[2]}(V){/}\Ki.$ 
\end{proof}
\begin{thm} \label{thm-t-4-4}
The $\Z$-graded Lie algebra 
$\at$ of \eqref{eq-t-4-11a}
is the maximal effective prolongation 
of 
type $\at_0(\Ki,\Li)$ of 
$\mt(\ft_{[2]}(V))$ 
and
the following are equivalent: \begin{itemize}
\item[$(i)$] 
$\gt(\Ki,\Li)$ is finite dimensional;
\item[$(ii)$] 
$\gt(\ft_{[2]}(V),\at_0(\Ki,\Li))$ is finite dimensional.
\end{itemize}
\end{thm} 
\begin{proof} Set $\at_0=\at_0(\Ki,\Li)$ and 
denote by $\bt{=}{\sum}_{\pq{\geq}{-}1}\!\bt_{\pq}$ the maximal 
EPFGLA
$\gt(\ft_{[2]}(V),\at_0)$ of type $\at_0$ of
$\at_{-1}=\ft(V){/}\ft_{[2]}(V)\simeq{V}.$ 
By construction, $\bt_0=\at_0.$ 
\par
Let $\piup:\mt\to(\mt{/}\nt)$ be the canonical
projection. Then $X\in\mt$ acts to the right  on $\bt$ by $\betaup{\cdot}X=[\betaup,\piup(X)]$ 
and
\begin{equation*}
 \mt\oplus\gt_0(\Ki,\Li)\oplus{\sum}_{\pq>0}\bt_{\pq}
\end{equation*}
is an effective prolongation of type $\Li$ of
$\mt(\Ki).$
Hence $\bt_{\pq}\subseteq\at_{\pq}$ for
$\pq{>}0.$ On the other hand, since $\at$ 
is an effective prolongation of type $\at_0$ of
$\mt(\ft_{[2]}(V)),$ 
we also have the opposite inclusion.
This yields $\bt_{\pq}=\at_{\pq}$ for all $\pq{>}0,$ 
proving the first part of the statement.
The equivalence of $(i)$ and $(ii)$ 
is then a consequence of Theorem~\ref{thm-t-4-2}.
\end{proof}

\subsection{Comparing maximal effective prolongations}
Let $V$ be a finite dimensional $\K$-vector space and 
$\Ki,\Ki'$ two graded cofinite
ideals in $\ft(V),$ so that 
\begin{equation*}
 \mt=\ft(V){/}\Ki,\;\;\;\mt'=\ft(V){/}\Ki'
\end{equation*}
are finite dimensional FGLA's. If $\Ki\,{\subset}\Ki',$ we get 
\begin{equation*}
 \mt\simeq\mt'\oplus\Ki'/\Ki .
\end{equation*}

\par 
By \eqref{eq-ta-1-12a}  the 
summands of positive degree of the 
maximal effective \mbox{$\Li${-}prolongation}
$\gt(\Ki,\Li)$
are 
\begin{align*}
\gt_p(\Ki,\Li)
&=\{X\in\ft_{p}(V,\Li)
\,\mid\,[X,\Ki]\subseteq\Ki\}.
\\[-8pt]
\end{align*}
If for the structure algebras we have the inclusion
$\Li\subseteq\Li'\subseteq\gl_\K(V),$ 
then the inclusions
$\ft_p(V,\Li)\subseteq\ft_p(V,\Li')$  
yield 
\begin{equation}\label{inclusion_prolongation}
 \gt_{p}(\Ki,\Li)\subseteq\gt_{p}(\Ki,\Li'),\;\;
 \;\forall p\geq{0}.
\end{equation}
\par 
We recall from Theorem~\ref{thm-t-4-4} that a maximal prolongation
$\gt(\Ki,\Li)$ 
is finite dimensional iff $\gt(\ft_{[2]}(V),\at_{0}(\Ki,\Li))$
 is finite dimensional. 
\par By Lemma \eqref{lem-tan-4-2} we have 
\begin{equation*}
\at_0(\Ki,\Li)
=\{X\in\Li\,\mid\,[X,\ft_{[2]}(V)]\subseteq\Ki\},\\
\end{equation*}
and hence when 
$\Li\,{\subseteq}\Li'\subseteq\gl_{\K}(V)$ 
and 
$\Ki\,{\subseteq}\,\Ki'$  
we obtain
\begin{equation} 
\at_0(\Ki,\Li)\subseteq\at_0(\Ki',\Li').
\end{equation}
Then the characterisation given in the proof of Theorem~\ref{thm-t-4-4}
yields
\begin{equation*}
 \at_{\pq}(\Ki,\Li){=}\gt_{\pq}\left(\ft_{[2]}(V),\at_{0}(\Ki,\Li)\right)\subseteq\gt_{\pq}\left(\ft_{[2]}(V),\at_{0}(\Ki',\Li')\right){=}
 \at_{\pq}(\Ki',\Li'),
\end{equation*}
for $\forall p>0$.

\begin{prop}\label{comparing}
Let $V$ be a finite dimensional $\K$-vector space, $\Li,\Li'
{\subseteq}\gl_{\K}(V)$ structure
algebras and $\Ki,\Ki'$ two cofinite graded ideals of $\ft(V).$ Assume that 
\begin{equation*}
 \Li\subseteq\Li',\;\; \Ki\subseteq\Ki'.
\end{equation*}
Then $\gt(\Li,\Ki)$ is finite dimensional, 
if $\gt(\Li',\Ki')$ is finite dimensional
and $\gt(\Li',\Ki')$ is infinite dimensional
if $\gt(\Li,\Ki)$ is infinite dimensional.
\end{prop}
\begin{exam}
Let $\mt{(\Ki)}$ be  a FGLA over $\K=\C,\R$ of \textit{depth} $\mu \geq 3$ such that  $\dim\gt_{-1}{=}\dim\gt_{-3}{=}2.$ Then the maximal EPFGLA 
$\gt(\Ki,\Li)$ is finite dimensional for every $\Li\subseteq\mathfrak{gl}
(\gt_{-1})$.

In fact we use Proposition~\ref{comparing}  over $\C$   with a gradation of the Lie algebra 
$\textrm{Lie}(\Gf_2)$
of  the special group $\Gf_2$;  for $\K=\R$ similar results 
can be found using the Lie algebra of the split real form 
of $\mathbf{G}_2$. 
Denoting by $X_{0},X_{1}$ a linear basis of $\gt_{{-}1},$ we get
\begin{align*}
\gt_{-1}&=\langle X_{0},X_1\rangle_{\K}\\
\gt_{-2}&=\langle [X_{0},X_1]\rangle_{\K}\\
\gt_{-3}&=\langle  [[X_{0},X_1],X_0], [[X_{0},X_1],X_1]\rangle_{\K} \\
&\,\,\hdots
\end{align*}
After noticing that the maximal 
effective
prolongation of 
$\ft(\K^{2}){/}\ft_{{[4]}}(\K^{2})$ is finite dimensional and isomorphic to the Lie algebra of $\Gf_{2}$ 
(see e.g. \cite[p.29]{Tan70} or \cite[p.62]{Warhurst2007}), by
using Proposition~\ref{comparing},
we conclude that every EPFGLA $\gt(\Ki,\Li)$ of any type $\Li\,{\subseteq}\,\gl_{2}(\K)$ of 
an $\mt(\Ki)$ with $\Ki{\subseteq}\ft_{[4]}(\K^{2})$
is finite dimensional.
\end{exam}
\section{$\Li$-prolongations 
of graded Lie algebras of the first kind} \label{sec5a}
By Theorem~\ref{thm-t-4-4}, finiteness of the maximal $\Li$-prolongation
of an FGLA $\mt$ of any finite kind is equivalent to that of
the $\Li'$-prolongation of its first kind quotient $V{=}\mt{/}\nt,$ 
for a Lie subalgebra $\Li'$  of $\Li$ that can 
be computed in terms of $\mt$ and $\Li.$ 
It is therefore a key issue to establish a viable criterion for
 FGLA's of the first kind.\par
Using duality, we will translate questions on the maximal effective
prolongations of FGLA's of the first kind 
to questions of commutative algebra for 
 finitely generated modules over polynomial rings.
\par 
\smallskip
Let $V$ be a finite dimensional $\K$-vector space,  
that we will consider as a 
commutative Lie algebra over $\K,$ 
and 
\begin{equation}
\Symm(V^*)={\sum}_{\pq=0}^\infty\Symm_{\pq}(V^*)
\end{equation}   
the $\Z$-graded unitary associative algebra over $\K$ of symmetric multilinear forms on $V.$ 
Its   product  is described on homogeneous forms by
\begin{equation*}\begin{aligned}
 (\xiup\per\etaup)(\vq_1,\hdots,\vq_{\pq+\qq}){=}
 \frac{1}{(\pq{+}\qq)!}\!
 \sum_{\sigmaup\in\Sb_{\pq+\qq}}\xiup(\vq_{\sigmaup_1},\hdots,\vq_{\sigmaup_\pq})
 \cdot\etaup(\vq_{\sigmaup_{\pq+1}},\hdots,\vq_{\sigmaup_{\pq+\qq}}),\quad\\
 \forall \xiup\in\Symm_{\pq}(V^*),\;\forall
 \etaup\in\Symm_{\qq}(V^*),\; \forall
 \vq_1,\hdots,\vq_{\pq+\qq}\in{V}.
 \end{aligned}
\end{equation*}

The duality pairing 
\begin{equation}
V\times{V}^*\ni (\vq,\xiup)\longrightarrow \langle\vq\mid\xiup\rangle
\in\K
\end{equation}
extends to a degree-$({-}1)$-derivation $D_{\vq}$
of $\Symm(V^*),$ 
with 
\begin{equation}\left\{\begin{aligned}
(D_\vq\xiup)(\vq_1,\hdots,\vq_\pq)=(\pq{+}1)
{\cdot}\xiup(\vq,\vq_1,\hdots,\vq_\pq),\;
\qquad\\
\forall \vq,\vq_1,\hdots,\vq_\pq
\in{V},\;\forall\xiup\in\Symm_{\pq+1}(V^*).
\end{aligned}\right.
\end{equation}
\par 
The tensor product 
\begin{equation}\label{equ-3-4}
\Xx(V){=}\Symm(V^*)\otimes{V}{=}\sum_{\pq\geq{-1}}\Xx_{\pq}(V),
\;\;\text{with}\;\; \Xx_{\pq}(V){=}
\Symm_{\pq+1}(V^*)\otimes{V},
\end{equation}
is the maximal $\gl_{\K}(V)$-prolongation of the commutative Lie algebra $\mt=V$, where each vector in $V$
is considered as a homogeneous element of degree~$({-}1).$ We can 
identify $\Xx(V)$ with the space of
vector fields with polynomial coefficients on $V.$ 
The Lie product in $\Xx(V)$ is described, 
on rank one elements, by 
\begin{equation*}
[\xiup{\otimes}\vq,\etaup{\otimes}\wq]{=}(\xiup{\per}(D_{\vq}\etaup))
{\otimes}\wq{-}(\etaup{\per}(D_{\wq}\xiup)){\otimes}\vq,
\forall\xiup{\in}\Symm_{\pq}(V^*),
\forall\etaup{\in}\Symm_{\qq}(V^*),\forall
\vq,\wq\in{V}.
\end{equation*}
\par\smallskip
If $\at_0$ is any Lie subalgebra of $\gl_{\K}(V),$ 
then
the direct  
sum $V\oplus\at_0$ is a Lie subalgebra of 
the Abelian extension $V\oplus\gl_{\K}(V)$ of $\gl_{\K}(V).$
The maximal effective $\at_0$-prolongation of $V$ can be described
as a Lie subalgebra of $\Xx(V).$
  
\begin{prop}  \label{prop-3-3.1}
Let $\at_0$ be any Lie subalgebra of $\gl_{\K}(V)$ and 
 \begin{equation}\label{eq-t-4-14}
\at=V\oplus\at_0\oplus{\sum}_{\pq\geq{1}}\at_{\pq} 
\end{equation}
the 
maximal effective
prolongation
of type $\at_0$ of $V.$ Its
summands of positive degree are
\begin{equation}\label{eq-t-4-15}
\at_{\pq}=\{\xiup{\in}\Xx_{\pq}(V)
\mid
\{V{\ni}\vq\to\xiup(\underset{\pq\;\text{times}}{\underbrace{\wq,\hdots,\wq}},\vq){\in}{V}\}{\in}
\at_0,\;\forall \wq{\in}{V}\}.
\end{equation}
 \end{prop} \begin{proof} It is known (see e.g. 
 \cite[Ch.VII \S{3}]{Sternberg} or \cite[Ch.1 \S{5}]{Kob})
 that the elements of $\at_{\pq},$ for $\pq{\geq}1,$ are 
 the $\xiup\in\Symm_{\pq{+}1}(V^*)\otimes{V}$ for which 
\begin{equation*}
\{ V\ni\vq\longrightarrow \xiup(\vq_1,\hdots,\vq_{\pq},\vq)\in{V}\}\in\at_0,\;\;\forall \vq_1,\hdots,\vq_{\pq}\in{V}.
\end{equation*}
  Formula \eqref{eq-t-4-15} follows by polarization. 
 \end{proof}
 \begin{rmk} Denote by $C:\Symm_{{\pq}{+}1}(V^*)\otimes{V}\to\Symm_{\pq}(V^*)$ the contraction map.
The Casimir element
 $\cq$ of ${V}^*{\otimes}V$ is the sum ${\sum}_{i{=}1}^n\epi_i{\otimes}\eq_i,$ where $\eq_1,\hdots,\eq_n$ is
 any basis of $V$ and $\epi_1,\hdots,\epi_n$ its dual basis in $V^*.$  We note that 
\begin{equation*}
 C(\xiup{\cdot}\cq)=(\pq{+}1){\cdot}\xiup,\;\;\forall\xiup\in\Symm_{\pq}(V^*),
\end{equation*}
so that the symmetric right product by $(\pq{+}1)^{{-}1}{\cdot}\cq$ is a right inverse of the contraction. 
In particular, 
\begin{equation*}
 \Xx_{\pq}(V)=\Xx'_{\pq}(V)\oplus\Xx''_{\pq}(V),\;\text{with}\; 
\begin{cases}
\Xx'_{\pq}(V)=\ker\left(C:\Symm_{{\pq}{+}1}(V^*){\otimes}{V}{\to}
\Symm_{\pq}(V^*) \right),\\
\Xx''_{\pq}(V)=\{\xiup{\cdot}\cq\mid \xiup{\in}\Symm_{\pq}(V^*)\},
\end{cases}
\end{equation*}
is the decomposition of $\Xx_{\pq}$ into a direct sum of irreducible $\gl_{\K}(V)${-}modules.
 \par
Since the operators $D_{\vq},D_{\wq}$ on $\Symm(V^*)$ commute, 
$$\Xx'(V)=V\oplus\slt_\K(V)\oplus{\sum}_{\pq{\geq}1}\Xx'_{\pq}(V)$$
 is the maximal EPFGLA of type $\slt_{\K}(V)$ of~$V.$ 
 \end{rmk}
\begin{exam}\label{examp-3.3}
Let $V,W$ be finite dimensional vector spaces over $\K$ and 
$\bq{:}V{\times}V{\to}{W}$ 
a non degenerate symmetric bilinear form. \par  
\textsl{If $\at_0$ is the orthogonal Lie algebra 
$\ot_{\bq}(V),$ 
consisting of $X{\in}\gl_{\K}(V)$ such that
\begin{equation*}
\bq(X(\vq_1),\vq_2){+}\bq(\vq_1,X(\vq_2)){=}0, \;\;
\text{for all $\vq_1,\vq_2{\in}V,$}
\end{equation*} 
then $\at_{1}=0.$ }
\begin{proof}
 An element $\xiup{\in}\at_1$ is a map 
 $\xiup{:}{V}{\to}\ot_{\bq}(V)$ such that $\xiup(\vq_1)(\vq_2){=}
 \xiup(\vq_2)(\vq_1)$ for all $\vq_1,\vq_2{\in}{V}.$ 
 Then, for $\vq_1,\vq_2,\vq_3{\in}V,$ we have 
\begin{align*}
 \bq(\xiup(\vq_1)(\vq_2),\vq_3)=\,{-}
 \bq(\vq_2,\xiup(\vq_1)(\vq_3))=\, {-}\bq(\vq_2,\xiup(\vq_3)(\vq_1))
 =  \bq(\xiup(\vq_3)(\vq_2),\vq_1)\\ 
 =  \bq(\xiup(\vq_2)(\vq_3),\vq_1) =  \,{-}\bq(\vq_3,\xiup(\vq_2)(\vq_1)) =
 \, {-}\bq(\xiup(\vq_1)(\vq_2),\vq_3).
\end{align*}
Since we assumed that $\bq$ is non degenerate, this implies that $\xiup(\vq_1)(\vq_2)=0$
for all $\vq_2{\in}V$ and hence that $\xiup(\vq_1){=}0$ for all $\vq_1{\in}V,$ i.e. that $\xiup{=}0.$ 
\end{proof} 
\end{exam}  

 The action 
\begin{equation}
 (\xiup\otimes\vq){\cdot}\wq=
 (D_\wq\xiup)\otimes\vq,\;\;\forall \xiup\in\Symm(V^*),\;\forall
 \vq,\wq\in{V}
\end{equation}
defines on $\Xx(V){=}\Symm(V^*){\otimes}V$ the structure of a $\Z${-}graded right $V${-}module.
Its $\Z$-graded dual  
\begin{equation}
 \Xx^*(V)=\sum_{\pq=-1}^\infty \Xx^*_{-\pq}(V),\;\;\text{with}\;\; \Xx^*_{-\pq}(V)=\Symm_{\pq+1}(V)\otimes{V}^*
\end{equation}
has a natural dual structure of
left $V$-module (see \S\ref{sec-t-1}).
The dual action of $V$ on $\Xx^*(V)$ is left multiplication, which 
extends to left multiplication by elements of
$\Symm(V).$ \par 
By using the right-$V$-module structure of $\Xx(V),$ we can rewrite the summands
$\at_{\pq}$ in \eqref{eq-t-4-14} by
\begin{equation*}
 \at_{\pq}=\{\xiup\in\Xx_{\pq}(V)
 \mid \xiup{\cdot}\vq_1{\cdots}
 \vq_\pq{\in}\at_0,\;\forall\vq_1,\hdots,\vq_{\pq}
 {\in}V\}.
\end{equation*}
Since $\at$ 
is a right-$V$-submodule
of $\Xx(V),$ its graded
dual $\at^*$ is the 
quotient of the left $V$-module $\Xx^*(V)$ 
by the annihilator $\Mi$ 
of $\at$ in $\Xx^*(V).$  \par 
The duality pairing of $V$ and $V^*$ makes
$\gl_{\K}(V)=V{\otimes}V^*$ self-dual, its pairing 
being defined 
by the trace form 
\begin{equation*}
 \langle{X}\,|\,{Y}\rangle=\trac(X{\cdot}Y),
 \;\;\forall X,Y\in\gl_{\K}(V).
\end{equation*}
Let
\begin{equation}
 \at_0^0=\{X\in\gl_{\K}(V)\mid 
 \trac(X{\cdot}A)=0,\;\forall A\in\at_0\}
\end{equation}
be the annihilator
of $\at_0$ in $\gl_{\K}(V).$ 
Then the annihilator $\Mi$ of $\at$ in $\Xx^*(V)$ 
is  the graded left-$V$-module
\begin{equation}
 \Mi=\Symm(V){\cdot}\at_0^0\subseteq\Xx^*(V).
\end{equation}
\begin{prop}\label{prop-tan-3-3}
 The dual of the maximal effective $\at_0$-prolongation 
 $\at$ of $V$ 
 is the quotient module 
\begin{equation}\vspace{-20pt}
 \at^*=\Xx^*(V){/}\Mi.
\end{equation} \qed
\end{prop}
\begin{prop} Let $V$ be a finite dimensional $\K$-vector space and 
$\at_0$ a Lie subalgebra of $\gl_{\K}(V).$ Then the following are equivalent:
\begin{enumerate}
\item the maximal effective $\at_0$-prolongation of $V$ is finite
dimensional;
\item the $\Symm(V)$-module $\at^*$ is $(V)$-coprimary;
\item $\Symm_h(V){\per}\at_0^0=\Symm_{h{+}1}(V)\otimes{V}^*,\;\;\text{
for some $h{\geq}0.$}$ 
\end{enumerate}
\end{prop} \begin{proof}
We recall that $\at^*$ is said to be coprimary if, for each $\sq{\in}\Symm(V),$ the homothety
$\{\at^*{\ni}\,\etaup{\to}\sq{\cdot}\etaup\,{\in}\at^*\}$ is either nilpotent or injective. If this is the case,
the radical $\sqrt{\Ann(\at^*)}$ of the ideal $\Ann(\at^*){=}\{\sq{\in}\Symm(V)\,{|}\, \sq{\cdot}\at^*{=}\{0\}\}$  
is a prime ideal in $\Symm(V).$ [See e.g. \cite[p.8]{toug1972}.]
\par 
Being $\Z$-graded, $\at^*$ is finite dimensional if and only if all $\vq{\in}V$ define nilpotent homotheties
on $\at_0^*.$ This shows that (1) and (2) are equivalent. Finally, (3) is equivalent to the fact that 
$\at_0^*,$ and hence $\at,$ is finite dimensional.
\end{proof}

The advantage of using duality is to reduce the question 
about the finite dimensionality of the
maximal prolongation to an exercise 
on finitely generated modules over the ring of polynomials
with coefficients in $\K$ 
and eventually to linear algebra. 
\par\smallskip
Having fixed a basis $\xiup_1,\hdots,\xiup_n$ of $V^*$ 
we can identify $\Xx^*(V)$ to $\Az^n.$ 
Each element $X$ of $\at_0^0$ can be viewed as a column vector $X\vq$ of $\Az^n,$
whose entries are first degree polynomials in $V.$ By taking a set $X_1,\hdots,X_m$ 
of generators of $\at_0^0$ we obtain a matrix of homogeneous first degree polynomials 
\begin{equation}\label{eq-3.12}
 M(\vq)=(X_1\vq,\hdots,X_m\vq) \in V^{\, n{\times}m}\subset{\Az}^{n{\times}m},
\end{equation}
that we can use to give a finite type presentation of $\at^*$: 
\begin{equation} \label{eq-3.13}
\begin{CD}
 \Az^m @>{M(\vq)}>> \Az^n @>>> \at^* @>>>0.
\end{CD}
\end{equation}
\par\smallskip
\begin{thm}\label{thm-t-5-5} 
 Let $\Ji_0(M)$ be the ideal generated by the order  
 $n$ minor determinants of $M(\vq).$
 A necessary and sufficient condition for $\at$ 
 to be finite dimensional is that
 $\sqrt{\Ji_0(M)}=(V).$ 
\end{thm} 
\begin{proof}
 Indeed, the ideals $\Ji_0(M)$ and  
 $\Ann(\at^*)=\{f{\in}\Az\mid f{\cdot}\at^*{=}\{0\}\}$
 have the same radical
 (see e.g. \cite[ch.{I},\S{2}]{toug1972}).
\end{proof}
Let $\Fb$ be the algebraic closure of the ground field $\K$ 
and set $V_{(\Fb)}{=}\Fb{\otimes}_{\K}V.$ By taking the tensor product by $\Fb$
we deduce from \eqref{eq-3.13}
the exact sequence 
\begin{equation} \label{eq-3.14}
\begin{CD}
 \Symm(V_{(\Fb)})^m @>{M(\zq)}>> \Symm(V_{(\Fb)}) @>>> \Fb\,{\otimes}_{\K}\at^* @>>>0.
\end{CD}
\end{equation}
We observe that $\at^*$ is $(V)$-coprimary if and only if $\Fb\,{\otimes}_{\K}\!\at^*$ is
$(V_{(\Fb)})$-coprimary. Therefore
Theorem~\eqref{thm-t-5-5} translates into 
\begin{thm}\label{thm-t-4-6}
 A necessary and sufficient condition for the maximal prolongation
 $\at$ to be finite dimensional is that 
\begin{equation}
 \rank(M(\zq))=n=\dim(V),\;\;\forall \zq\in{V}_{(\Fb)}{\setminus}\{0\}.
\end{equation}
\end{thm} 
\begin{proof}
 In fact, since $\Fb$ is algebraically closed, 
 by the Nullstellensatz (see e.g. \cite{Hart})
 the necessary and sufficient condition
 for an ideal $\Ji$ of $\Symm(V_{(\Fb)})$ to have $\sqrt{\Ji}{=}(V_{(\Fb)})$ 
 is that 
$\{\zq{\in}V_{(\Fb)}{\mid} \, f(\zq){=}0,\;\forall f\in\Ji\}{=}\{0\}.$ 
\end{proof}
\begin{exam} \label{ex-tan-3-6}
Denote by
$\mathfrak{co}(n,\K){=}\{A{\in}\gl_n(\K)\,{\mid}\, 
A^\intercal{+}A{\in}\K{\cdot}\Id_n\}$  
the Lie algebra of conformal transformations of $\K^n.$ 
\paragraph{(1)} 
 Let $\at_0=\mathfrak{co}(2,\K).$  Its orthogonal in $\gl_2(\K)$ 
 consists of the traceless
 symmetric matrices. Take the basis consisting of 
\begin{equation*}
 M_1= 
\begin{pmatrix}
 1 & 0\\
 0 & {-}1
\end{pmatrix},\;\; M_2= 
\begin{pmatrix}
 0 & 1\\
 1 & 0
\end{pmatrix}.
\end{equation*}
We obtain 
\begin{equation*}
 M(\zq)= 
\begin{pmatrix}
 \zq_1 & \zq_2\\
 {-}\zq_2 & \zq_1
\end{pmatrix},\;\; \text{with}\;\; \det(M(\zq))=\zq_1^2{+}\zq_2^2.
\end{equation*}
Since the equation $\zq_1^2{+}\zq_2^2{=}0$ has non zero solutions in $\Fb^2,$ 
by Theorem~\ref{thm-t-4-6},
$\K^2$ has an infinite dimensional
effective $\mathfrak{co}(2,\K)$-prolongation. 
\paragraph{(2)}
Let us consider next the case $n{>}2.$ The orthogonal $\at_0^0$ 
of $\at_0{=}\mathfrak{co}(n,\K)$ consists of the
traceless $n{\times}n$ symmetric matices. As a basis of $\at_{0}^0$ we can take the matrices\par
\centerline{$\Delta_h{=}(\deltaup_{1,i}\deltaup_{1,j}{-}\deltaup_{h,i}\deltaup_{h,j}),$  $(h{=}2,\hdots,n)$ 
and $T_{h,k}{=}(\deltaup_{h,i}\deltaup_{k,j}{+}\deltaup_{k,i}\deltaup_{h,j}),$  $(1{\leq}{h}{<}k{\leq}n).$}
\par
Accordingly, we get
\begin{equation*}
 M(\zq)= 
\begin{pmatrix}
\zq_1 & \zq_1 & \hdots & \zq_1 & \zq_2 & \zq_3 & \hdots & \zq_n & 0 & \hdots \\
{-}\zq_2 & 0 & \hdots & 0 & \zq_1 & 0 & \hdots & 0 & \zq_3 & \hdots \\
0 & {-}\zq_3 & \hdots & 0 & 0 & \zq_1 & \hdots & 0 & \zq_2 & \hdots \\
\vdots & \vdots & \ddots & \vdots & \vdots & \vdots & \ddots & \vdots&\vdots & \hdots \\
0 & 0 & \hdots &{-}\zq_n & 0 & 0 & \hdots & \zq_1 & 0 & \hdots 
\end{pmatrix}.
\end{equation*}
We want to show that $M(\zq)$ has rank $n$ when $\zq{\in}\Fb^n{\setminus}\{0\}.$ 
The minor of the 
$(j{-}1)$-st, $n$-th, $\hdots,$ $(n{-}2)$-nd columns
is  
$(\zq_1^2{+}\zq_j^2){\cdot}\zq_1^{n{-}2},$ for $j{=}2,\hdots,n.$  
Likewise,  
we can show that the ideal of order $n$ minor determinants
of $M(\zq)$  
contains all polynomials $(\zq_j^2{+}\zq_h^2){\cdot}\zq_j^{n{-}2}$
for $1{\leq}j{\neq}h{\leq}n.$ Denote by ${\pm}\iq$ the roots of $({-}1)$ in $\Fb$ and assume by contradiction
that $M(\zq)$ has rank ${<}n$ for some $\zq{\neq}0.$ We can assume that $\zq_1{\neq}0.$
This yields $\zq_{j}\,{=}{\pm}\,\iq\,\zq_1$ for $j{=}2,\hdots,n.$ But then 
\begin{equation*}
 (\zq_2^2{+}\zq_3^2)\cdot\zq_2^{n{-}2}=-2(\pm\iq)^{n{-}2}\zq_1^n\neq{0}.
\end{equation*}
This contradiction 
proves that $M(\zq)$ has rank $n$ for all $\zq{\in}\Fb^n{\setminus}\{0\},$ showing, 
by Theorem~\ref{thm-t-4-6}, that the maximal EPFGLA of type
$\mathfrak{co}(n,\K)$ of $\K^n$ is finite dimensional if $n{\geq}3.$ 
\end{exam}
\begin{exam} Let us fix integers $0{<}\pq{\leq}\qq$ and set $\Bt=\left.\left\{ 
\begin{pmatrix}
 0 & B^{\intercal}\\
 B & 0
\end{pmatrix}\,\right|  B{\in}\K^{\pq{\times}\qq}\right\}.
$ 
Set $n{=}\pq{+}\qq$ and let
\begin{equation*}
 \at_0=\{X\in\gl_{\K}(n)\mid X^\intercal B{+}B\,X\in\Bt,\;\forall{B}\in\Bt\}
\end{equation*}
be the Lie algebra of $\Bt$-conformal tranformations of $\K^n.$ We have 
\begin{equation*}
 \at_0{=}\left.\left\{ 
\begin{pmatrix}
 X & 0\\
 0 & Y 
\end{pmatrix}\,\right| X{\in}\gl_{\K}(\qq),Y{\in}\gl_{\K}(\pq)\right\}\;\text{and hence}\;
\at_0^0= \left.\left\{ 
\begin{pmatrix}
0 & E\\
F^{\intercal} & 0
\end{pmatrix}\right| E,F{\in}\K^{\qq{\times}\pq}\right\}.
\end{equation*}
By taking the canonical basis of $\K^{\qq{\times}\pq}$ we obtain 
\begin{equation*}
 M(\zq)= 
\begin{pmatrix}
 \zq_{\qq{+}1}\Id_{\qq} & \hdots & \zq_n\Id_{\qq} & 0 & \hdots & 0 \\
 0 & \hdots& 0 & \zq_1\Id_{\pq} & \hdots & \zq_{\qq}\Id_{\pq}
\end{pmatrix}.
\end{equation*}
Clearly both $(\zq_1,\hdots,\zq_{\pq})$ and $(\zq_{\pq{+}1},\hdots,\zq_n)$ are associated ideals
of $\Ji(M)$ and hence the maximal effective $\at_0$-prolongation of $V$ is infinite dimensional.
Note that, if we take  the $\Bt$-orthogonal Lie algebra 
\begin{equation*}
 \ot_{\Bt}=\{X\in\gl_{\K}(n)\mid X^\intercal{B}+B\,X=0,\;\forall B{\in}\Bt\},
\end{equation*}
 the maximal $\ot_{\Bt}$-prolongation of $\K^n$ is finite dimensional by Example~\ref{examp-3.3},
because, if $B_1,\hdots,B_{\pq\qq}$ is a basis of the $\K$-vector space $\Bt,$ the symmetric
bilinear form \;
$\K^n{\times} \K^n{\ni}(\vq_1,\vq_2)\to(\vq_1^{\intercal}B_i\vq_2)_{1{\leq}i{\leq}\pq\qq}
\in\K^{\pq\qq}$ \; is nondegenerate. 
\end{exam}
\begin{exam}\label{ex-tan-3.9}
 Let $V=\K^{2n}$ and 
\begin{align*}
 \at_0&=\spt(n,\K)=\{A\in\gl_{2n}(\K)\mid A^\intercal\Omega+\Omega{A}=0\},\;\;\text{with}\;\;
 \Omega= 
\begin{pmatrix}
 0 & \Id_n\\
 {-}\Id_n& 0
\end{pmatrix},\\
&=\left.\left\{ 
\begin{pmatrix}
 A & B \\
 C & -A^\intercal
\end{pmatrix}\right| A,B,C\in\gl_{n}(\K),\; B^\intercal{=}{B},\; C^\intercal{=}C\right\}.
\end{align*}
Then 
\begin{equation*}
 \at_0^0=\left.\left\{ 
\begin{pmatrix}
 A & B \\
 C & A^\intercal
\end{pmatrix}\right| A,B,C\in\gl_n(\K),\; B^{\intercal}{=}{-}B,\; ,C^\intercal{=}{-}C\right\}.
\end{equation*}
Take any basis $A_1,\hdots,A_h$ of $\gl_n(\K)$ (with $h{=}n^2$) and 
$B_1,\hdots,B_k$ of $\ot(n,\K)$ (with $k{=}\tfrac{1}{2}n(n{-}1)$). Then we have,
for $\zq,\wq\in\K^n,$
\begin{equation*}
 M(\zq,\wq)= 
\begin{pmatrix}
 A_1\zq & \hdots & A_h\zq & B_1\wq & \hdots & B_k\wq & 0 & \hdots & 0\\
 A_1^\intercal\wq & \hdots & A_h^\intercal\wq & 0 & \hdots & 0 & B_1\zq & \hdots & B_k\zq 
\end{pmatrix}.
\end{equation*}
If we take e.g. $\wq{=}0,$ we see that, for $\zq_0{\neq}0,$ 
$M(\zq_0,0)$ has rank $(2n{-}1),$ because $B_1\zq_0,\hdots,B_k\zq_0$
span the \textit{orthogonal} hyperplane $\zq_0^\perp{=}\{\zq{\in}\K^n{\mid}\,\zq^\intercal\zq_0{=}0\}$
to $\zq_0$ 
in $\K^n$. 
By Theorem~\ref{thm-t-4-6} this implies that the maximal
$\spt(n,\K)$-pro\-lon\-ga\-tion of $\K^{2n}$ is infinite dimensional.
\end{exam}
The criterion of Theorem~\ref{thm-t-4-6} 
can also be expressed, in an equivalent way,
as an \textit{ellipticity condition}.
In the case
when, given an $\mt{=}{\sum}_{0{<}\pq{\leq}\muup}{\gt}_{-\pq},$
we take $\Li{=}\gl_{\K}(V)$ and hence
$\at_{0}=\{\xiup\,{\in}\,\gl_{\K}(V)
\,{\mid}\,[\xiup,\gt_{-\pq}]\,{=}\,0,\;\text{for}\;\pq{>}1\},$ it 
was already
considered by many authors: see e.g. 
\cite{GQS1966,Spencer1969} and \cite{Ottazzi2011} for 
a thorough discussion.

\begin{thm}\label{thm-3.11}
 Let $V$ be a finite dimensional $\K$-vector space, $\at_0$  a Lie subalgebra of $\gl_{\K}(V)$ 
and $\Fb$ the algebraic closure of $\K.$ \par
The maximal EPFGLA of type $\at_0$ of $V$ is infinite dimensional if and only if
$\Fb{\otimes}\at_0$ contains an element of rank one on $V_{(\Fb)}.$ 
\end{thm} 
\begin{proof}
We can as well assume that $\K$ is algebraically closed. \par
All rank one elements of $\gl_\K(V)$ can be written in the form $\vq{\otimes}\xiup,$
with nonzero $\vq{\in}V$ and $\xiup{\in}V^*.$
If $X{\in}\gl_{\K}(V),$ we obtain 
\begin{equation}\tag{$*$} \label{eq-3.*}
 \langle{X}\vq\,{|}\, \xiup\rangle = \trac(X\,(\vq{\otimes}\xiup)).
\end{equation}
If $\at_0$ contains $\vq{\otimes}\xiup,$ then $X\vq{\in}\xiup^0{=}\{\wq{\in}V{\mid} \langle\wq\,{|}\,\xiup
\rangle{=}0\}$ for all $X{\in}\at_0^0,$ showing that $M(\vq)$ has rank less than $n{=}\dim(V).$ 
Then $\dim(\at){=}\infty$ by Theorem~\ref{thm-t-4-6}.\par
Vice versa, if $M(\vq)$ has rank less than $n$ for some $\vq{\neq}0,$ we can find a
nonzero $\xiup{\in}V^*$ such that $X\vq{\in}\xiup^0$ for all $X{\in}\at_0^0.$ Since the trace form
is nondegenerate on $\gl_{\K}(V),$ by \eqref{eq-3.*} this implies that $(\vq{\otimes}\xiup)$
is a rank one element of~$\at_0.$ 
\end{proof}
\subsection{Prolongation of irreducible representations}
\label{subs-3-3}
 Assume that $\K$ is algebraically closed and let $V$ be a finite dimensional faithful irreducible
 representation of a reductive Lie algebra $\at_0$ over $\K,$ having a center $\zt_0$ of dimension
 ${\leq}1.$ 
 Let $\st=[\at_0,\at_0]$ be the semisimple ideal of $\at_0,$  
 $\hg$ its Cartan subalgebra 
 and $\Rad,$ $\Lambda_W$ its corresponding root system and
 weight lattice. 
 Fix a lexicographic order on $\Rad,$ corresponding to the choice of a Borel subalgebra
 of $\st.$  Then
 $V$ is a faithful irreducible $\st$-module. Let $\Lambda(V){\subset}\Lambda_W$ be the
 set of its weights and
$\phiup$ its dominant weight.
If $\psiup$ is minimal in $\Lambda(V),$ then $\phiup{-}\psiup$ is  dominant 
in $\Lambda(V{\otimes}V^*).$ Let $\vq_\phiup$ be a maximal vector in $V$
and $\xiup_{{-}\psiup}$ a maximal covector in $V^*.$ Then $v_{\phiup}\otimes\xiup_{{-}\psiup}$
is an element, of rank one on $V,$ 
 generating an irreducible $\st$-sub-module $L_{\phiup{-}\psiup}$ 
 of $V{\otimes}V^*.$ 
  By a Theorem of Dynkin (see \cite{As94}, \cite[Ch.XIV]{Cahn}, \cite{Dy1952} ) 
  we know that all  nonzero elements
 of $L_{\phiup{-}\psiup}^0$ have rank larger than one. Hence by using
 Theorem~\ref{thm-3.11},
 we obtain 
\begin{thm}\label{thm-3.12}
 The maximal effective
 prolongation 
of type $\at_0$ of $V$
 is infinite dimensional if and only if \par\centerline{$L_{\phiup{-}\psiup}{\subset}\st.$}
\end{thm}
\begin{proof}
The Lie algebra $\st$ decomposes into a direct sum $S_1{\oplus}\cdots{\oplus}S_k$ 
of irreducible $\st$-sub-modules of
$V{\otimes}V^*.$ The summands $S_i$ which are distinct from $L_{\phiup{-}\psiup}$
are contained in $L_{\phiup{-}\psiup}^0.$ The statement follows from this observation.
\end{proof}
Under the assumptions above, if $\st{=}[\at_0,\at_0]$ is simple, then
the maximal prolongation $\at$ of type $\at_0$ of $V$
is primitive in the sense 
explained in  
\cite{guillemin1970}.
Then Theorem~\ref{thm-3.12} yields easily 
a result about the infinite dimensionality 
of the maximal effective \textit{primitive}
prolongations
that has been already 
proved 
by several Authors
(see \cite{Cartan1909, guillemin1970, 
GQS1966, Guillemin1967, Kac1967, kn1965b,
morimoto1970,  shnider1970,  
Wil1971}). \par 
\begin{prop} \label{prop-3.11}
Assume that $\K$ is algebraically closed,
that $\at_0$ is reductive, with $[\at_0,\at_0]$ simple, 
and that $V$ is a faithful irreducible $\at_0$-module. Then
the maximal effective prolongation of type 
$\at_0$ of $V$ is infinite dimensional if and only if
one of the following is verified 
\begin{itemize}
 \item[$(i)$] 
 $\at_0$ is equal either to $\gl_{\K}(V)$ or $\slt_{\K}(V)$;
 \item[$(ii)$] $V{\simeq}\K^{2n}$ 
 for some integer $n{\geq}2$ and $\at_0$ is isomorphic either to
 $\spt(n,\K)$ or to $\mathfrak{c}\spt(n,\K).$ \qed
\end{itemize}
\end{prop}

\section{$\gl_{\K}(V)$-prolongations of 
FGLA's
of the second kind} \label{sec5}
Let $\mt{=}\gt_{{-}1}\oplus\gt_{{-}2}$ be an FGLA of the second kind. 
Set $V{=}\gt_{{-}1}.$ 
By Proposition~\ref{prop-tan-1.4}, $\mt$ is isomorphic to a quotient
$\ft(V){/}\Ki,$ for a graded ideal $\Ki$ of $\ft(V)$ with $\ft_{[3]}(V){\subseteq}\Ki{\subsetneqq}\ft_{[2]}(V).$ 
We note that 
$\Ki{=}\Ki_{\;{-}2}{\oplus}\ft_{[3]}(V)$ is a graded ideal of $\ft(V)$ 
for  every
proper vector subspace $\Ki_{\;{-}2}$ of 
$\ft_{-2}(V).$ \par 
Since
 $\ft_{-2}(V){=}\Lambda^2(V),$ 
the subspace
$\Ki_{\;-2}$ is the kernel 
 of the surjective linear map 
 $\lambdaup:\Lambda^2(V)\rightarrow\!\!\!\!\rightarrow\gt_{-2}$ 
associated to the bilinear map   
$(\vq_1,\vq_2)\to [\vq_1,\vq_2]$ defined by the Lie product of elements of $V$:
\begin{equation}\label{eq2.5a}
 \xymatrix{{V}\times{V} \ar[rr]^{(\vq_1,\vq_2)
 \to\vq_1{\wedge}\vq_2} \ar[dr]_{(\vq_1,\vq_2)\to[\vq_1,\vq_2]\quad}
 && \Lambda^2({V}) \ar[dl]^{\lambdaup}\\
 &\,\gt_{{-}2}.}
\end{equation}
\begin{exam}
 Let us consider $n+\binom{n}{2}$ variables, 
 that we label by $x_i$ for $1{\leq}i{\leq}n$ and
 $t_{i,j}$ for $1{\leq}i{<}j{\leq}n.$   
 The algebra $\mt$ of real vector fields 
 with polynomial coefficients generated by 
\begin{equation*}
 \gt_{-1}\simeq\R^n=\left\langle 
 \frac{\partial}{\partial{x}_i}+{\sum}_{h<i}x_h\,
 \frac{\partial}{\partial{t}_{h,i}}
 -{\sum}_{h>i}x_h\,\frac{\partial}{\partial{t}_{i,h}}
 \mid 1{\leq}i{\leq}n\right\rangle
\end{equation*}
has 
$\gt_{-2}{=}[\gt_{-1},\gt_{-1}]=\left\langle 
\frac{\partial}{\partial{t}_{i,j}}\mid 1{\leq}i{<}j{\leq}n
\right\rangle$
so that $\mt\simeq\R^n\oplus\Lambda^2(\R^n),$ 
corresponds to the case $\Ki_{\;-2}=\{0\}.$ 
\end{exam}
We keep the notation \eqref{eq-1.13}.
For te case of FGLA's of the second kind 
Proposition~\ref{prop-tan-1.4} reads
\begin{prop}
Two graded fundamental Lie algebras of the second kind 
 $\mt(\Ki)$ and $\mt(\Ki')$ 
 are isomorphic if and only if the homogeneous parts of second degree 
 $\Ki_{\;{-2}}$ and $\Ki'_{\;{-2}}$ of $\Ki$ and $\Ki'$ 
 are $\GL_{\K}(V)$-congruent.
 \qed
\end{prop} 
Fix  
a proper vector subspace $\Ki_{\;{-}2}$ of $\ft_{{-}2}(V){=}\Lambda^2(V)$ and 
the corresponding $2$-cofinite ideal 
$\Ki{=}\Ki_{\;{-}2}{\oplus}\ft_{[3]}(V).$ 
By \eqref{eq-ta-1-12a}, the Lie algebra of zero-degree derivations of $\mt(\Ki)$ is characterised by
\begin{lem}
 \label{lem-tan-6-3} 
The Lie algebra  
$\gt_0(\Ki)$ of the $0$-degree  derivations 
of $\mt(\Ki)$ is 
\begin{equation}\vspace{-20pt}
\label{eq2.3} 
 \gt_0(\Ki)=\{A\in\gl_\K(V)\mid T_A(\Ki_{\;-2})\subseteq\Ki_{\;-2}\}.
\end{equation}\qed
\end{lem}
We recall from \S\ref{sec-fund} that, for $\Ki{=}\Ki_{\;-2}{\oplus}\ft_{[3]}(V),$ 
the maximal effective canonical $\gl_\K(V)$-prolongation of $\mt(\Ki)$ 
is the $\Z${-}graded Lie algebra
\begin{equation}
\gt(\Ki)={\sum}_{\pq\geq{-}2}\gt_{\pq}(\Ki)\end{equation} 
whose homogeneous summands 
are defined by 
\begin{equation} 
\begin{cases}
\gt_{-2}(\Ki){=}\Lambda^2(V){/}\Ki_{\;-2},\\
\gt_{-1}(\Ki){=}V,\\
\gt_h(\Ki)=\Der_h\left(\mt,{\sum}_{\pq<h}\gt_{\pq}(\Ki)\right),\;\;\text{for $h{\geq}0.$} 
\end{cases}
\end{equation}
The spaces
$\gt_h(\Ki)$ are defined by recurrence and 
consist of the  degree $h$ 
homogeneous derivations of $\mt(\Ki)$ with values in the $\mt(\Ki)$-module 
${\sum}_{\pq<h}\gt_{\pq}(\Ki)$: 
an element $\alphaup$ 
of $\gt_h(\Ki)$ is identified with a map 
\begin{gather*}
\alphaup:(V,\Lambda^2(V))\longrightarrow
(\gt_{h-1}(\Ki),\gt_{h-2}(\Ki)),\;\; \text{with}\\
\alphaup(\vq{\wedge}\wq)=
(\alphaup(\vq))(\wq){-}(\alphaup(\wq))(\vq),\;\;\forall \vq,\wq{\in} 
V,\;\;\text{and}\;\;
\alphaup(\omegaup)=0,\;\;\forall\omegaup\in\Ki_{\;-2}.
\end{gather*}
By Theorem~\ref{thm-t-4-2}, $\gt(\Ki)$ is finite dimensional if and only if 
\begin{equation}
 \at(\Ki)\,{=}\!\!\sum_{\pq=-1}^\infty \at_{\pq}(\Ki),\;\text{with}\; \at_{\pq}(\Ki){=}\{
 \alphaup{\in}\gt_{\pq}(\Ki){\mid} [\alphaup,\gt_{-2}(\Ki)]{=}\{0\}\}
\end{equation}
is finite dimensional.
\begin{lem} 
We have \;
$\at_0(\Ki_{\;-2})=\{A\in\gl_n(\R)\mid T_A(\Lambda^2(\R^n))\subseteq\Ki_{\;-2}\}.$
\qed
\end{lem}
\par\smallskip
Fix an identification $V{\simeq}\K^n,$ to consider the non degenerate symmetric bilinear form
$\bq(\vq,\wq)=\vq^{\intercal}{\cdot}\wq$ on $V.$ It yields an isomorphism
\begin{equation}
 \rhoup:\Lambda^2(V)\longrightarrow\ot(V),\;\;\text{with}\;\;
 \rhoup(\vq{\wedge}\wq)=\vq{\cdot}\wq^{\intercal}-\wq{\cdot}\vq^\intercal,\;\forall\vq,\wq\in{V},
\end{equation}
between $\Lambda^2(V)$ and the orthogonal Lie algebra 
\begin{equation}
 \ot(V)=\{X\in\gl_{\K}(V)\mid X^{\intercal}{+}\,X=0\}.
\end{equation}
Under this identification, the action of $A{\in}\gl_{\K}(V)$ on $\Lambda^2(V)$ 
can be described by $A{\cdot}X{=}A\,X{+}X\,A^{\intercal}.$ 
We use $\rhoup$ to identify 
$\Ki_{\;-2}$ with a linear subspace of $\ot(V).$  
With this notation, we introduce
\begin{equation}\label{eq2.4}
\Ki_{\;-2}^\perp=\{X\in\ot(V)\mid \trac(X{\, }K)=0,\;\forall K\in\Ki_{\;-2}\}\simeq\gt_{-2}(\Ki).
\end{equation}

\begin{rmk} For the ideal $\Ki'=\Ki_{\;-2}^\perp\oplus\
\ft_{[3]}(V)$ we have
\begin{equation}
\gt_0(\Ki')=\{A^\intercal\mid A\in\gt_0(\Ki)\}.
\end{equation}
which is a Lie algebra anti-isomorphic to  
$\gt_0(\Ki),$ but, of course, $\at_0(\Ki)$ and $\at_0(\Ki')$ may turn out to be quite different. \par
Note that $\Ki^\perp$ can be canonically identified with $\gt_{{-}2}.$ 
\end{rmk}
\begin{lem}\label{lem-t-6-14}
 We have \;
$ \at_0^0(\Ki)=\{X{\, }Y\mid X\in\ot(V),\; Y\in\Ki_{\;-2}^\perp\}.$
\end{lem} 
\begin{proof}
 In fact, for $A{\in}\gl_{\K}(V)$ and $X,Y{\in}\ot(V),$ we obtain 
\begin{align*}
 \trac((AX{+}XA^\intercal)Y)=\trac(AXY)+\trac(YXA^{\intercal})=2\trac(A{\, }(XY)),
\end{align*}
because $(YXA^\intercal)^\intercal=A{\, }(XY).$ Thus we obtain 
\begin{align*}
 A\in\at_0\Leftrightarrow \trac((A{\, }X{+}X{\, }A^\intercal){\, }Y){=}
2\trac(A{\, }(Y{\, }X)) =0, \;\;
\forall X{\in}\ot(V),\forall{Y}{\in}\Ki^\perp,
\end{align*}
proving our statement.
\end{proof}

To apply Theorem~\ref{thm-t-4-6} to investigate the finite dimensionality of the
maximal $\gl_{\K}(V)$-prolongation of $\mt(\Ki),$ 
we need 
to construct the matrix $M(\zq)$ in \eqref{eq-3.14}, relative to $\at_0^0(\Ki),$ that we will
denote by $M_2(\Ki,\zq).$  As usual, we denote by $\Fb$ be the algebraic closure of $\K$
and set $V_{(\Fb)}=\Fb{\otimes}_{\K}V.$ \par 
 We can proceed as follows. 
Fix generators
$Y_1,\hdots,Y_m$ of $\Ki^\perp$ and, for ${{\zq}}{\in}V_{(\Fb)}$  and $1{\leq}i{\leq}m,$ 
consider the vectors $Y_i{{\zq}}{\in}{V}_{(\Fb)}.$ By the identification $V{\simeq}\K^n,$ 
\begin{equation}\label{eq-t-6-21}
\Phi_{\Ki}(\zq)=(Y_1\zq,\hdots,Y_m\zq)
\in
(\K[\zq_1,\hdots,\zq_n])^{n\times{m}},
\end{equation}
is a matrix of
first order homogeneous  polynomials in $\K[\zq_1,\hdots,\zq_n]$ 
and, after choosing generators $X_1,\hdots,X_N$ of $\ot(V)$ as a $\K$-linear space,
 we take 
\begin{equation}\label{equ-tan-4.16}
M_2(\Ki,\zq)=(X_1\Phi_{\Ki}(\zq),\hdots,X_N\Phi_{\Ki}(\zq))\in
(\K[\zq_1,\hdots,\zq_n])^{n\times(mN)}.
\end{equation}

The $\ot(V)$-orbit of a non zero vector $\zq$ of $\Fb^n$ 
spans the hyperplane 
\begin{equation*}
\zq^\perp=\{\wq\in\Fb^n\mid \zq^{\intercal}{}\,
\wq=0\}.
\end{equation*} 
To check this fact, we can reduce to the case where $\Fb{=}\K.$ 
Take any $\uq{\in}V$ with $\zq^{\intercal}\uq{=}1.$ If $\wq{\in}\zq^\perp,$ then the matrix 
$X{=}\wq\,\uq^\intercal{-}\uq\,\wq^{\intercal}$ belongs to $\ot(V)$ and
$X\,\zq{=}\wq.$ This shows that $\zq^{\perp}{\subset}\,\ot(V)\,\zq.$ The opposite
inclusion is obvious, since $\zq^{\intercal}X\zq{=}0$ for all $X{\in}\ot(V)$ and $\zq{\in}V.$
\par 
Hence, if $\zq_1,\zq_2{\in}V_{(\Fb)}$ are
linearly independent, then $\ot(n)\,\zq_1{+}\ot(n)\,\zq_2$ spans $V_{(\Fb)},$ 
so that
$M_2(\Ki,\zq)$ has rank $n$ for all
$\zq$ for which $\Phi_{\Ki}(\zq)$ has rank ${\geq}{2}.$  
We proved the following 
\begin{prop} \label{prop-t-6-10}
Let $\Ki_{\;-2}\subset\ft_{-2}(V)$ and $\Ki=\Ki_{\;-2}\oplus\ft_{[3]}(V).$ Then
the maximal $\gl_{\K}(V)$-prolongation 
$\gt(\Ki)$ 
of $\mt(\Ki)$
is finite dimensional if and only if 
\begin{equation}\vspace{-20pt}
\{\zq\in{V}_{(\Fb)}\mid \rank(\Phi_{\Ki}(\zq))<{2}\}=\{0\}.
\end{equation}\qed
\end{prop} 
We can give an 
equivalent formulation of Proposition~\ref{prop-t-6-10} 
involving the \textit{rank} of the $\Fb$-bilinear extension of the
alternate bilinear form on $V$ defined by the Lie brackets.
\begin{defn} Let $\Fb$ be the algebraic closure of $\K.$ We call the integer 
\begin{equation}
 \ell=\inf\{\dim_{\Fb}([\zq,V_{(\Fb)}])\mid 0\neq\zq\in{V}_{(\Fb)}\}
\end{equation}
 the \emph{algebraic minimum rank} of $\mt(\Ki).$ 
\end{defn}
\begin{thm}\label{thm-tan-4-8}
Let $\Ki_{\;-2}\subset\ft_{-2}(V)$ and $\Ki=\Ki_{\;-2}\oplus\ft_{[3]}(V).$ Then
 the maximal $\gl_{\K}(V)$-prolongation $\gt(\Ki)$ of $\mt(\Ki)$ is finite dimensional if and only if $\mt(\Ki)$
has algebraic minimum rank $\ell\geq{2}.$ 
\end{thm} 
\begin{proof} 
If we identify $\zq\,{\in}{V}_{(\Fb)}$ with the corresponding numerical vector in $\Fb^n,$ 
 then $\wq^{\intercal}Y_i\zq$ is the $Y_i$-component of $[\wq,\zq].$ 
 The condition that the minimum rank of $\mt(\Ki)$ is larger or equal to two is then equivalent to the fact that
 $\Phi_{\Ki}(\zq)$ has rank ${\geq}2$  for all $\zq\in{V}_{(\Fb)}.$ 
\end{proof}
\begin{exam} Let $\K=\R,$ 
$n{=}4$ and $\Ki=\left.\left\langle 
\left(\begin{smallmatrix}
 0 & A\\
 {-}A & 0
\end{smallmatrix}\right)\right| A\in\R^{2\times{2}}, A=A^\intercal\right\rangle.$\par
 With $\Jd{=} 
\left(\begin{smallmatrix}
 0 & 1\\
 {-}1 & 0
\end{smallmatrix}\right),$ a basis of $\Ki^\perp$ is given by the matrices 
\begin{equation*}
 Y_1= 
\begin{pmatrix}
 \Jd & 0\\
 0 & 0
\end{pmatrix},\;\; Y_2= 
\begin{pmatrix}
 0 & 0\\
 0 & \Jd
\end{pmatrix},\;\; Y_3=   
\begin{pmatrix}
 0 & \Jd\\
 \Jd & 0
\end{pmatrix}.
\end{equation*}
 This yields 
\begin{equation*}
 \Phi_{\Ki}(\zq) = 
\begin{pmatrix}
 \zq_2 & 0 & \zq_4\\
 {-}\zq_1 & 0 & {-}\zq_3\\
 0 & \zq_4 & \zq_2\\
 0 & {-}\zq_3 & {-}\zq_1
\end{pmatrix}.
\end{equation*}
Writing $\Delta^{i,j}_{h,k}$ for the minor of the lines $i,j$ and the columns $h,k$, we obtain 
\begin{equation*}
\Delta^{2,4}_{1,3}=\zq_1^2,\; \; \Delta^{1,3}_{1,3}=\zq_2^2,\;\;\Delta^{2,4}_{1,3}=\zq_3^2,\;\;
 \Delta^{1,3}_{2,3}=\zq^2_4,
\end{equation*}
showing that $\gt(\Ki)$ is finite dimensional.
\end{exam}
\begin{exam} Let $\K=\R,$ 
 $n{=}4$ and $\Ki^\perp$ generated by the matrices 
\begin{equation*}
 Y_1=
\begin{pmatrix}
 0 & 0 & 1 & 0\\
 0 & 0 & 0 & 1\\
 {-}1 & 0 & 0 & 0\\
 0 & {-}1 & 0 & 0
\end{pmatrix},\; Y_2= 
\begin{pmatrix}
 0 & 0 & 0 & 1\\
 0 & 0 & {-}1 & 0\\
 0 & 1 & 0 & 0\\
 {-}1 & 0 & 0 & 0
\end{pmatrix},\; Y_3= 
\begin{pmatrix}
 0 & 1 & 0 & 0 \\
 {-}1 & 0 & 0 & 0 \\
 0 & 0 & 0 & 0\\
 0 & 0 & 0 & 0
\end{pmatrix}.
\end{equation*}
Then 
\begin{equation*}
 \Phi_{\Ki}(\zq)= 
\begin{pmatrix}
 \zq_3 & \zq_4 & \zq_2\\
 \zq_4 & {-}\zq_3 & {-}\zq_1\\
 {-}\zq_1 & \zq_2 & 0 \\
 {-}\zq_2 & {-}\zq_1 & 0 
\end{pmatrix}.
\end{equation*}
The matrix $\Phi_{\Ki}(\zq)$ has rank ${\geq}2$ for all nonzero $\zq\in\R^4,$ but $$\Phi_{\Ki}(0,0,1,\iq)=\left( 
\begin{matrix}
 1 & \iq & 0\\
 \iq & {-}1 & 0\\
 0 & 0 & 0\\
 0 & 0 & 0
\end{matrix}\right)$$ has rank $1$.
Thus $\gt(\Ki)$ is infinite dimensional, although $\mt(\Ki)$ has (real) minimal rank $2.$
\end{exam}
We note that, if $\Phi_{\Ki}(\zq)$ has at most $2$ columns, then the set 
$$\{\zq{\in}V_{(\Fb)}{\mid} \rank(\Phi_{\Ki}(\zq)){<}2\}$$
is an algebraic affine variety of positive dimension in $V_{(\Fb)}.$ In particular, we obtain 
(cf. e.g. \cite{Kruglikov2011})
\begin{cor}\label{cor-tan-4-11}
 If $\dim(\gt_{-2}(\Ki)){\leq}{2},$ then $\gt(\Ki)$ is infinite dimensional. \qed
\end{cor}
\section{$\gl_{\K}(V)$-prolongations of 
FGLA's
of higher kind}\label{sect7}
Let $V$ be a $\K$-vector space of finite dimension
$n{\geq}2$ and,  
for an integer $\muup\geq{3},$ 
fix a $\muup$-cofinite $\Z$-graded ideal $\Ki{\subset}\ft_{[2]}(V)$ of  
$\ft(V).$ 
Let
\begin{equation}
\mt(\Ki){=}{\sum}_{\pq=-1}^{-\muup}\gt_{\pq}(\Ki)\simeq\ft(V){/}\Ki , \quad   
\nt(\Ki){=}{\sum}_{\pq={-}{2}}^{-\muup}\gt_{{\pq}}(\Ki)\simeq\ft_{[2]}/\Ki.
\end{equation}
The canonical  maximal $\Z$-graded effective 
$\gl_{\K}(V)$-prolongation 
\begin{equation*}
 \gt(\Ki)={\sum}_{\pq\geq{-}\muup}\gt_{\pq}(\Ki)
\end{equation*}
of $\mt(\Ki)$ is the quotient by $\Ki$ of its normaliser $\Nt(\Ki)$ in $\Ft(V).$ 
By Tanaka's criterion (Theorem~\ref{thm-t-4-4}), $\gt(\Ki)$ is finite dimensional if, and only if, 
\begin{align*}
 \at(\Ki)={\sum}_{\pq\geq{-}1}\at_{\pq}(\Ki),\quad\!\!\text{with}\; \begin{cases}\at_{-1} =\mt(\Ki){/}\nt(\Ki)\simeq{V},
 \\ \at_{\pq}(\Ki) {=}
 \{X{\in}\gt_{\pq}(\Ki)\,{\mid}\, [X,\nt(\Ki)]{=}\{0\}\}\,\;\text{for}\; \pq{\geq}0,\end{cases}
\end{align*}
is finite dimensional.
As in \S\ref{sec5} we get (see also Lemma~\ref{lem-tan-4-2}):
\begin{prop}\label{prop-t-6-3} 
Let $\Ki$ be a $\Z$-graded
ideal of $\ft(V),$ contained in $\ft_{[2]}.$
Then
\begin{equation}\vspace{-18pt}
\begin{cases}
\gt_0(\Ki)=\{A\in\gl_{\R}(V)\mid T_A(\Ki)\subset\Ki\},\\
 \at_0(\Ki)=\{A\in\gl_{\R}(V)\mid T_A(\ft_{-\pq}(V))\subseteq\Ki_{\; -\pq},\;\forall\pq\,{\geq}{2}\}.
 \end{cases} 
\end{equation}
\qed
\end{prop}
Set (brackets are computed in $\ft(V)$) 
\begin{equation}
 W(\Ki)=\{\vq{\in}V\mid [\vq,\ft_{[2]}(V)]\subset\Ki\}.
\end{equation}  
\begin{prop}\label{prop-t-6-5}
We have the following characterisation: 
\begin{equation}\label{eq-5.4}
\at_0(\Ki)=\{A\in\gl_{\K}(V)\mid T_A(\ft_{{-}2}(V))\subseteq\Ki_{\;-2},\;
A(V)\subseteq{W}(\Ki)\}.
\end{equation}
\end{prop} 
\begin{proof}
We use the characterisation in Proposition~\ref{prop-t-6-3} and the fact that
$V$ generates $\ft(V).$ 
Let us denote by $\bt$ 
the right hand side of \eqref{eq-5.4}. \par
We first show that $\bt{\subseteq} \at_0(\Ki).$
To this aim, we need to check that, if $A{\in}\bt,$ then $T_A(\ft_{{-}\pq}){\subseteq}\Ki_{\;{-}{\pq}}$ for 
$\pq{\geq}2.$ This is true by assumption if $\pq{=}2.$ Then we argue by recurrence. We have 
\begin{equation}\label{eq-star5} \tag{$*$}
 T_A([\vq,X])=[A(\vq),X]+[\vq,T_A(X)], \;\;\;\forall A{\in}\gl_{\K}(V),\;\forall\vq{\in}V,\;\forall{X}{\in}\ft(V).
\end{equation}
If $T_A(\ft_{{-}\pq}){\subseteq}\Ki_{\;{-}{\pq}}$ for some $\pq{\geq}2,$ then, for $X{\in}\ft_{{-}\pq}(V)$
and $\vq{\in}V,$ 
the second summand $[\vq,T_A(X)]$
in the right hand side belongs to $\Ki_{\;{-}p{-}1}$ because $T_A(X){\in}\Ki_{\;{-}\pq}$ and 
$\Ki$ is an ideal. If $A{\in}\bt,$ then
$[A(\vq),X]$
belongs to $\Ki_{\;{-}\pq{-}1}$ by the assumption that $A(\vq){\in}W(\Ki).$ 
Since the elements of the form $[\vq,X]$ with $\vq{\in}V$ and $X{\in}\ft_{{-}\pq}(V)$ generate
$\ft_{{-}\pq{-}1}(V),$ this shows that $T_A(\ft_{{-}\pq{-}1}){\subseteq}\Ki_{\;{-}\pq{-}1},$ 
completing  the proof of the inclusion $\bt{\subseteq}\at_0(\Ki).$ \par 
Let us prove the opposite inclusion. 
If $A{\in}\at_0(\Ki),$ 
then both the left hand side and the second summand of the right hand side of \eqref{eq-star5}
belong to $\Ki_{\;{-}\pq{-}1}$ for all $\vq{\in}V$ and $X{\in}\ft_{{-}\pq}(V).$ 
This shows that $[A(\vq),\ft_{{-}\pq}(V)]{\subseteq}\Ki_{\;{-}\pq}$ for all $\pq{\geq}2$ and hence
that $A(V){\subseteq}W(\Ki).$ Thus
$\at_0(\Ki){\subseteq}\bt,$ completing the proof of the proposition.
\end{proof}
Let us fix a basis $\eq_1,\hdots,\eq_n$ whose first $m$ vectors $\eq_1,\hdots,\eq_m$ are a
basis of $W(\Ki).$ The matrices of the elements $A$ of $\gl_{\K}(V)$ that map $V$ into $W(\Ki)$
are of the form 
\begin{equation*}
 A= 
\begin{pmatrix}
 B  \\
 0 
\end{pmatrix},\;\;\text{with}\;\; B{\in}\K^{m{\times}n}.
\end{equation*}
Their orthogonal in $\gl_{\K}(V)$ consists of the matrices of the form 
\begin{equation*}
 Y = 
\begin{pmatrix}
 0 & C
\end{pmatrix},\;\;\text{with}\;\; C{\in}\K^{n\times(n{-}m)}.
\end{equation*}
Taking the canonical basis $Y_1,\hdots,Y_r,$ with $r{=}n{\times}(n{-}m),$ for the linear space of the
matrices of this form, we obtain a matrix 
\begin{equation}\label{eq-5.5}
 M_3(\Ki,\zq)=(\zq_{m{+}1}\Id_n,\hdots,\zq_n\Id_n).
\end{equation}
Since the orthogonal of an intersection of linear subspaces is the sum of their orthogonal
subspaces, the matrix of the form
\eqref{eq-3.12} associated to $\at_0(\Ki)$ is in this case 
\begin{equation}\label{eq-5.6}
M(\zq)=
 (M_2(\Ki,\zq),M_3(\Ki,\zq)).
\end{equation}
 Taking into account that $M_3(\Ki,\zq)$ is described by \eqref{eq-5.5}, the matrix
\eqref{eq-5.6} has rank $n$ whenever $(\zq_{m{+}1},\hdots,\zq_n){\neq}(0,\hdots,0).$ 
Let $\Phi_{\Ki}(\zq)$ the matrix defined by \eqref{eq-t-6-21}.
Then Theorem~\ref{thm-t-4-4} yields
\begin{thm} \label{thm-ta-7.3}
Assume that $\Ki$ is $\muup$-cofinite for some integer
$\muup{\geq}2$ and let $\eq_1,\hdots,\eq_m$ be a basis of $W(\Ki).$ Then  
$\gt(\Ki)$  is finite dimensional if and only if $\Phi_{\Ki}(\zq)$ has rank ${\geq}2$ 
for all $0{\neq}\zq{\in}W_{(\Fb)}=\Fb{\otimes}W(\Ki).$
 \qed
\end{thm} 
\section{$\Li$-prolongations of FGLA's of general kind}\label{sect8}
Let $\Li$ be a Lie subalgebra of $\gl_{\K}(V).$ Fix a set $A_1,\hdots,A_k$ 
of generators of 
\begin{equation*}
 \Li^\perp=\{A{\in}\gl_{\K}(V)\mid \trac(AX)=0,\;\forall X{\in}\Li\}
\end{equation*}
and form 
\begin{equation}
 M_1(\Li,\zq)=(A_1\zq,\hdots,A_k\zq),\;\;\zq\in{V}_{(\Fb)}.
\end{equation}
If we fix a  basis $\eq_1,\hdots,\eq_n$ of $V,$ we may consider $M(\zq)$ as a matrix 
of homogeneous first degree polynomials of $\zq{=}(\zq_1,\hdots,\zq_n)$ with coefficients in~$\K.$ \par 
Let 
$\mt={\sum}_{\pq=1}^\muup\gt_{{-}\pq}$ 
be an FGLA of finite kind $\muup{\geq}1$ and set $V=\gt_{{-}1}.$
Then $\mt$ is isomorphic to a quotient $\ft(V)/\Ki,$ for a graded ideal $\Ki$ of $\ft(V)$ with 
$\ft_{[\muup{+}1]}(V){\subseteq}\Ki{\subseteq}\ft_{[2]}(V).$
Let $\Li$ be a Lie subalgebra of $\gl_{\K}(\gt_{{-}1}).$    We showed in~\S\ref{sec-fund}
that the maximal effective prolongation of type $\Li$ of $\mt$ is the quotient
by $\Ki$ of its normaliser $\Nt(\Ki,\Li)$ in $\Ft(V,\Li).$ 
 \par 
 By Theorem~\ref{thm-t-4-4}, this maximal effective
$\Li$-prolongation $\gt(\Ki,\Li)$ of $\mt$ 
is finite dimensional if
 and only if the maximal effective $\at_0$-prolongation of $V$ 
 is finite dimensional, where $\at_0$ the Lie subalgebra  of $\Li$ defined by
\begin{equation}
 \at_0{=}\{A{\in}\Li \mid T_A(\ft_{[2]}(V))\subseteq\Ki\}.
\end{equation}
Since the orthogonal of an intersection of linear subspaces is the sum of their orthogonal
subspaces, 
we can use 
the results (and the notation)
of \S\ref{sec5a},\ref{sec5},\ref{sect7} to construct the matrix $M(\zq)$ of
Theorem~\ref{thm-t-4-6}, obtaining 
\begin{equation}\label{eq-6.3}
 M(\zq)=(M_1(\Li,\zq),M_2(\Ki,\zq),M_3(\Ki,\zq)). 
\end{equation}
Let $W{=}W(\Ki)$ be the subspace of $V$ defined in \S\ref{sect7}.
Then we obtain  
\begin{thm}\label{thm-tan-6.1}
Fix the basis $\eq_1,\hdots,\eq_n$ of $V$ in
such a way that the first $m$ vectors $\eq_1,\hdots,\eq_m$
form a basis of $W(\Ki).$ Then
 a necessary and sufficient condition in order that the maximal effective 
 $\Li$-prolongation of $\mt$ be finite dimensional is that 
\begin{equation*}
(M_1(\Li)(\zq),M_2(\Ki,\zq))
\end{equation*}
has rank $n$ for all $\zq\in{W}_{(\Fb)}{\setminus}\{0\}
=\{\zq\in\Fb^n{\setminus}\{0\}\mid \zq_i{=}0,\;\forall i{>}m\}.$ \qed
\end{thm}
\begin{exam}
 Let $V$ be a finite dimensional $\K$-vector space of dimension $n{\geq}3$ 
 and $\bq:V{\times}V{\to}\K$ 
 a nondegenerate bilinear form on $V,$ having nonzero symmetric and antisymmetric 
 components $\bq_s$ and $\bq_a.$ Let 
\begin{equation*}
 \Li=\{X{\in}\gl_{\K}(V)\mid \exists \cq(X){\in}\K\;\text{s.t.}\;
 \bq(X\vq,\wq){+}\bq(\vq,X\wq){=}\cq(X)\bq(\vq,\wq),\;\forall\vq,\wq{\in}B\}.
\end{equation*}
 Then it is natural to take $\mt{=}\gt_{{-}1}{\oplus}\gt_{{-}2},$ with $\gt_{{-}1}{=}V$ and
 $\gt_{{-}2}{=}\K,$ defining the Lie brackets on $V$ by 
\begin{equation*}
 [\vq,\wq]=\bq_a(\vq,\wq),\;\;\forall\vq,\wq{\in}V.
\end{equation*}
Then $\mt{\simeq}\ft(V){/}\Ki,$ with $\Ki{=}\Ki_{{-}2}{\oplus}\ft_{[3]}$ for 
\begin{equation*}
\Ki_{{-}2}{=}\left.\left\{{\sum}\vq_i{\wedge}\wq_i{\in}\Lambda^2(V)\right|  
{\sum}\bq_a(\vq_i,\wq_i){=}0\right\}.
\end{equation*}
The maximal effective $\Li$-prolongation
$\gt{=}\gt(\Ki,\Li)$ of $\mt$ is equal to the maximal effective $\mathfrak{co}_{\bq_s}(V)$-prolongation
$\gt(\Ki,\mathfrak{co}_{\bq_s}(V))$ of $\mt.$ \par
In particular, $\gt$ is finite dimensional by Example~\ref{ex-tan-3-6}, because $n{\geq}3,$
when $\bq_s$ is nondegenerate. \par
Consider now the case where $\bq_s$ has rank $0{<}\pq{<}n.$ We can find a basis 
$\eq_1,\hdots,\eq_n$ of $V$ such that $\bq_s$ is represented by a matrix 
\begin{equation*}
 B= 
\begin{pmatrix}
 D & 0\\
 0 & 0
\end{pmatrix},\;\; \text{with}\;\; D= \diag(\lambdaup_1,\hdots,\lambdaup_{\pq})\in\K^{\pq{\times}\pq},
\;\text{with}\; \lambdaup_1{\cdots}\lambdaup_\pq{\neq}0.
\end{equation*}
Let $A$ be the antisymmetric $n{\times}n$ matrix representing $\bq_a$ in this basis and 
$A_1,\hdots,A_n$ its rows. Then the matrix $M(\zq)$ of \eqref{eq-6.3} has the form 
\begin{equation*}
 M(\zq)=(M_1(\zq),M_2(\zq)),
\end{equation*}
with $M_1(\zq)=(\zq_1{E}_1,\hdots,\zq_\pq{E}_{\pq})$ for invertible
$n{\times}n$ matrices  $E_1,\hdots,E_q$ and 
\begin{equation*}
 M_2(\zq)= 
\begin{pmatrix}
 A_2\zq & \hdots & A_n\zq & 0 & \hdots & 0 & \hdots \\
 {-}A_1\zq & \hdots & 0 &  A_{3}\zq & \hdots & A_n\zq &\hdots \\  
 0 & \hdots & 0 & {-}A_2(\zq) & \hdots & 0 & \hdots \\
 \vdots & \ddots & \vdots &\vdots & \ddots &\vdots & \hdots\\
 0 & \hdots & {-}A_1\zq & 0 & \hdots & {-}A_2\zq&\hdots\\
\end{pmatrix}
\end{equation*}
In particular, when $\zq_i{=}0$ for $1{\leq}i{<}n,$ the first row of $M(\zq)$ is zero
because $A$ is antisymmetric and therefore $\gt$ is infinite dimensional by 
Theorem~\ref{thm-tan-6.1} (here $W_{(\Fb)}=V_{(\Fb)}$ because $\Ki_{{-}3}{=}\ft_{[3]}(V)$).
\end{exam}

\providecommand{\bysame}{\leavevmode\hbox to3em{\hrulefill}\thinspace}
\providecommand{\MR}{\relax\ifhmode\unskip\space\fi MR }

\providecommand{\MRhref}[2]{%
  \href{http://www.ams.org/mathscinet-getitem?mr=#1}{#2}
}
\providecommand{\href}[2]{#2}


\begin{thebibliography}{10}

\bibitem{alek}
D.~Alekseevsky and L.~David, \emph{Prolongation of {T}anaka structures: an
  alternative approach}, Ann. Mat. Pura Appl. (4) \textbf{196} (2017), no.~3,
  1137--1164. \MR{3654947}

\bibitem{AMN06}
A.~Altomani, C.~Medori, and M.~Nacinovich, \emph{The {CR} structure of minimal
  orbits in complex flag manifolds}, J. {L}ie Theory \textbf{16} (2006), no.~3,
  483--530. \MR{MR2248142 (2007c:32043)}

\bibitem{As94}
H.~Aslaksen, \emph{Determining summands in tensor products of lie algebra
  representations}, Journal of Pure and Applied Algebra \textbf{93} (1994),
  no.~2, 135--146.

\bibitem{Bour89}
N.~Bourbaki, \emph{Commutative algebra}, Springer-Verlag, Berlin, 1989.

\bibitem{n1998lie}
\bysame, \emph{Lie groups and {L}ie algebras. {C}hapters 1--3}, Elements of
  Mathematics (Berlin), Springer-Verlag, Berlin, 1989, Translated from the
  French, Reprint of the 1975 edition. \MR{979493}

\bibitem{Bou68}
\bysame, \emph{Lie groups and {L}ie algebras. {C}hapters 4--6}, Elements of
  Mathematics (Berlin), Springer-Verlag, Berlin, 2002, Translated from the 1968
  French original by Andrew Pressley. \MR{1890629}

\bibitem{Cahn}
R.N. Cahn, \emph{Semi-simple {L}ie algebras and their representations},
  Frontiers in Physics, vol.~59, Benjamin/Cummings, 1984.

\bibitem{Cartan1909}
E.~Cartan, \emph{Les groupes de transformations continus, infinis, simples},
  Annales scientifiques de l'cole Normale Suprieure \textbf{26} (1909),
  93--161 (fre).

\bibitem{Dy1952}
E.~B. Dynkin, \emph{Maximal subgroups of the classical groups}, Trudy Moskov.
  Mat. Ob\v{s}\v{c} \textbf{1} (1952), 39--166 (russian).

\bibitem{guillemin1970}
V.W. Guillemin, \emph{Infinite dimensional primitive {L}ie algebras}, J.
  Differential Geom. \textbf{4} (1970), no.~3, 257--282.

\bibitem{GQS1966}
V.W. Guillemin, D.~Quillen, and S.~Sternberg, \emph{The classification of the
  complex primitive infinite pseudogroups}, Proceedings of the National Academy
  of Sciences of the United States of America \textbf{55} (1966), no.~4,
  687--690.

\bibitem{Guillemin1967}
\bysame, \emph{The classification of the irreducible complex algebras of
  infinite type}, Journal d'Analyse Math{\'e}matique \textbf{18} (1967), no.~1,
  107--112.

\bibitem{GS}
V.W. Guillemin and S.Sternberg, \emph{An algebraic model of transitive
  differential geometry}, Bull. Amer. Math. Soc. \textbf{70} (1964), 16--47.

\bibitem{Hart}
R.~Hartshorne, \emph{Algebraic geometry}, Springer-Verlag, New York-Heidelberg,
  1977, Graduate Texts in Mathematics, No. 52. \MR{0463157}

\bibitem{Kac1967}
V.~G. Kac, \emph{Simple graded {L}ie algebras of finite height}, Funkcional.
  Anal. i Prilo\v zen \textbf{1} (1967), no.~4, 82--83. \MR{0222126}

\bibitem{Kob}
S.~Kobayashi, \emph{Transformations groups in differential geometry},
  Ergebnisse der Mathematik und ihrer Grenzgebiet, 70, Springer-Verlag, Berlin,
  1972.

\bibitem{kn1964a}
S.~Kobayashi and T.~Nagano, \emph{On filtered {L}ie algebras and geometric
  structures {I}}, Journal of Mathematics and Mechanics \textbf{13} (1964),
  no.~5, 875--907.

\bibitem{kn1964}
\bysame, \emph{A theorem on filtered {L}ie algebras and its applications},
  Bull. Amer. Math. Soc. \textbf{70} (1964), no.~3, 401--403.

\bibitem{kn1965}
\bysame, \emph{On a fundamental theorem of {W}eyl-{C}artan on
  {$G$}-structures}, J. Math. Soc. Japan \textbf{17} (1965), no.~1, 84--101.

\bibitem{kn1965a}
\bysame, \emph{On filtered {L}ie algebras and geometric structures {II}},
  Journal of Mathematics and Mechanics \textbf{14} (1965), no.~3, 513--521.

\bibitem{kn1965b}
\bysame, \emph{On filtered {L}ie algebras and geometric structures {III}},
  Journal of Mathematics and Mechanics \textbf{14} (1965), no.~4, 679--706.

\bibitem{kn1966}
\bysame, \emph{On filtered {L}ie algebras and geometric structures, {IV}},
  Journal of Mathematics and Mechanics \textbf{15} (1966), no.~1, 163--175.

\bibitem{Kruglikov2011}
B.~Kruglikov, \emph{Finite-dimensionality in {T}anaka theory}, Annales de
  l'Institut Henri Poincar\'e (C) Non Linear Analysis \textbf{28} (2011),
  no.~1, 75--90.

\bibitem{NMSM}
S.~Marini, C.~Medori, M.~Nacinovich, and A.~Spiro, \emph{On transitive contact
  and {$CR$} algebras}, arXiv:1706.03512v1 [math.DG] (2017), to appear
  in Annali della Scuola Normale Superiore di Pisa - Classe di Scienze,
  {DOI Number: 10.2422/2036-2145.201710\_012}

\bibitem{morimoto1970}
T.~Morimoto and N.~Tanaka, \emph{The classification of the real primitive
  infinite {L}ie algebras}, J. Math. Kyoto Univ. \textbf{10} (1970), no.~2,
  207--243.

\bibitem{Ottazzi2010}
A. Ottazzi,
\emph{A sufficient condition for nonrigidity of Carnot groups.)}
Math. Z.  \textbf{259} (2008), no. 3, 617--629. 

\bibitem{Ottazzi2011a}
A. Ottazzi, B.  Warhurst, 
\emph{Algebraic prolongation and rigidity of Carnot groups.}
Monatsh. Math. \textbf{162} (2011), no. 2, 179--195. 

\bibitem{Ottazzi2011}
A. Ottazzi, B.  Warhurst,  
\emph{Contact and 1-quasiconformal maps on Carnot groups.} 
J. Lie Theory \textbf{21} (2011), no. 4, 787--811.

\bibitem{Reu1993}
C.~Reutenauer, \emph{Free {L}ie algebras}, London Mathematical Society
  Monographs. New Series, vol.~7, The Clarendon Press, Oxford University Press,
  New York, 1993, Oxford Science Publications. \MR{1231799}

\bibitem{shnider1970}
S.~Shnider, \emph{The classification of real primitive infinite lie algebras},
  J. Differential Geom. \textbf{4} (1970), no.~1, 81--89.


\bibitem{Spencer1969}
D.C.~Spencer,
\emph{Overdetermined systems of linear partial differential equations},
Bull. Amer. Math. Soc. \textbf{75} (1969), 179--239. 


\bibitem{Sternberg}
S.~Sternberg, \emph{Lectures on differential geometry}, Chelsea Publ.\ Co., New
  York, 1983, First edition: Prentice Hall, Inc.\ Englewood Cliffs, N.J., 1964.

\bibitem{Tan67}
N.~Tanaka, \emph{On generalized graded {L}ie algebras and geometric structures.
  {I}}, J. Math. Soc. Japan \textbf{19} (1967), 215--254. \MR{MR0221418 (36
  \#4470)}

\bibitem{Tan70}
\bysame, \emph{On differential systems, graded {L}ie algebras and
  pseudogroups}, J. Math. Kyoto Univ. \textbf{10} (1970), 1--82. \MR{MR0266258
  (42 \#1165)}

\bibitem{toug1972}
J.C. Tougeron, \emph{Ideaux de fonctions differentiables}, Ergebnisse der
  Mathematik und ihrer Grenzgebiete. 2. Folge, Springer Berlin Heidelberg,
  1972.

\bibitem{Warhurst2007}
B.~Warhurst, \emph{Tanaka prolongation of free {Lie} algebras}, Geometriae
  Dedicata \textbf{130} (2007), no.~1, 59--69.

\bibitem{Wil1971}
R.L. Wilson, \emph{Irreducible {L}ie algebras of infinite type}, Proc. Amer.
  Math. Soc. (1971), no.~29, 243--249.
  

\end{thebibliography}
\end{document}